\documentclass{article}
\usepackage{amsmath,amssymb,url,color,mathtools}
\usepackage{amsmath,amssymb,amsthm,mathtools}
\usepackage[margin=3cm]{geometry}
\usepackage{tikz}
\usepackage{bbold}
\usepackage{authblk}

\newcommand{\qbin}[2]{\genfrac{[}{]}{0pt}{}{#1}{#2}}

\renewcommand{\L}{\mathcal{L}}
\newcommand{\bb}{\mathbb}
\newcommand{\mrm}{\mathrm}
\newcommand{\mf}{\mathfrak}

\newcommand{\sym}{\mathrm{Sym}}
\newcommand{\alt}{\mathrm{Alt}}
\newcommand{\psl}{\mathrm{PSL}}
\newcommand{\pgl}{\mathrm{PGL}}

\newcommand{\Der}{\mathrm{Der}}
\newcommand{\der}{\mathrm{der}}
\newcommand{\Cay}{\mathrm{Cay}}
\newcommand{\fix}{\mathrm{fix}}

\newcommand{\Ome}{\Omega}

\newcommand{\ch}{\mathbf{v}}

\newcommand{\onevec}{ {\bf{1}} }
\newcommand{\one}{ {\mathbb{1}} }
\newcommand{\ind}{ \mathrm{ind}}
\newcommand{\res}{ \mathrm{res}}

\DeclareMathOperator{\PG}{PG}
\DeclareMathOperator{\AG}{AG}

\newtheorem{lemma}{Lemma}[section]
\newtheorem{prop}[lemma]{Proposition}
\newtheorem{cor}[lemma]{Corollary}
\newtheorem{thm}[lemma]{Theorem}
\newtheorem{example}[lemma]{Example}
\newtheorem{question}[lemma]{Question}
\newtheorem{notation}[lemma]{Notation}

\newtheorem{remark}[lemma]{Remark}

\theoremstyle{definition}
\newtheorem{mydef}[lemma]{Definition}

\usepackage{stmaryrd}
\usepackage{trimclip}

\makeatletter
\DeclareRobustCommand{\shortto}{%
  \mathrel{\mathpalette\short@to\relax}%
}

\newcommand{\short@to}[2]{%
  \mkern2mu
  \clipbox{{.5\width} 0 0 0}{$\m@th#1\vphantom{+}{\shortrightarrow}$}%
  }
\makeatother

\begin{document}

\title{Cameron-Liebler sets in permutation groups} 

\author[1]{Jozefien D'haeseleer\ \thanks{Jozefien.Dhaeseleer@ugent.be}}
\affil{Department of Mathematics: Analysis, Logic and Discrete Mathematics, Ghent University, Belgium}
\author[2]{Karen Meagher\ \thanks{karen.meagher@uregina.ca}}
\author[2]{Venkata Raghu Tej Pantangi\ \thanks{pvrt1990@gmail.com}}
\affil[2]{Department of Mathematics and Statistics, University of Regina,	Regina, Saskatchewan S4S 0A2, Canada}

	


\maketitle

\begin{abstract}
Consider a group $G$ acting on a set $\Omega$, the vector $\ch_{a,b}$ is a vector with the entries indexed by the elements of $G$, and the $g$-entry is 1 if $g$ maps $a$ to $b$, and zero otherwise.
A \textsl{$(G,\Omega)$-Cameron-Liebler set} is a subset of $G$, whose indicator function is a linear combination of elements in $\{\ch_{a, b}\ :\ a, b  \in \Omega\}$. 
We investigate Cameron-Liebler sets in permutation groups, with a focus on constructions of Cameron-Liebler sets for 2-transitive groups.
\end{abstract}

\section{Introduction}
\subsection{Some history on Cameron-Liebler sets}
The investigation of Cameron-Liebler sets, or, in short, \textsl{CL sets}, of permutation groups is inspired by the research on Cameron-Liebler sets in finite projective spaces. 
In \cite{CL}, Cameron and Liebler introduced special line classes in $\PG(3,q)$ when investigating the orbits of the subgroups of the collineation group of $\PG(3,q)$. It is well known, by Block’s Lemma \cite[Section 1.6]{block}, that a collineation group PGL$(n+1,q)$ of a finite projective space $\PG(n, q)$ has at least as many orbits on lines as on points. Cameron and Liebler tried to determine which collineation groups have equally many point and line orbits. They found that the line orbits of the subgroups with equally many orbits on lines
and points, fulfill many (equivalent) combinatorial and algebraic properties. A set of lines, fulfilling one, and hence all of these properties, was later called a Cameron-Liebler set of lines in $\PG(3,q)$.

 More precisely, a set $\L$ of lines in $\PG(3, q)$, with characteristic vector $\ch_{\L}$, is a CL set of lines if and only if  $\ch_{\L}\in V_0\perp V_1$, where $V_0$ and $V_1$ are the first two eigenspaces of the related Grassmann scheme. Moreover, if $A$ is the point-line incidence matrix of $\PG(3,q)$, then it can be shown that a set $\L$ is a CL set of lines if and only if  $\ch_{\L}\in$ $\mathrm{im}(A^T)$. Note that a column of $A^T$ corresponds to a point $P$ in $\PG(3,q)$, and that this column is the characteristic vector of the set of all lines through $P$. This set of lines is often called a \textsl{point-pencil} or \textsl{canonical example} of a set of pairwise intersecting lines in $\PG(3,q)$. From this observation, it follows that CL sets can be seen as a linear combination of these canonical examples of intersecting families of lines. 

{We see that there is a strong link between intersecting families and CL sets. This connection will continue to hold true in the context of group theory, see Section \ref{SectionEKRCL}.}\\

The examination of Cameron-Liebler sets in projective spaces motivated the definition and investigation of Cameron-Liebler sets of generators in polar spaces~\cite{CLpol}, Cameron-Liebler classes in finite sets \cite{CLset} and Cameron-Liebler sets of $k$-spaces in $\PG(n,q)$ and in $\AG(n,q)$ \cite{ikCLpol,ikCLaffien}. Furthermore, Cameron-Liebler sets can be introduced for any distance-regular graph. This has been done in the past under various names: Boolean degree $1$ functions~\cite{Ferdinand.}, completely regular codes of strength $0$ and covering radius $1$~\cite{Mogilnykh}, tight sets of type I~\cite{MR2679936}. We refer to the introduction of~\cite{Ferdinand.} for an overview.

The main questions, independent of the context where Cameron-Liebler sets are investigated, are always the same: What are the possibilities for the size of a CL set, and what CL sets exist with a given size?
We will partially solve this question for Cameron-Liebler sets in the context of permutation groups.

\subsection{EKR and CL sets in permutation groups}\label{SectionEKRCL}

Throughout this paper we use $G$ to denote a group with a transitive action on a set $\Omega$. 
Given a subset $S \subset G$, we denote the \textsl{indicator function} of $S$ by $\ch_S$, this is the element in $\mathcal{C} [G]$ that is the sum of all the elements in $S$. This can also be viewed a vector, known as the characteristic vector---the entries of $\ch_S$ are indexed by the elements in $G$ and the entry corresponding to $g \in G$ is 1 if $g \in S$ and $0$ otherwise. 
Given $a,b \in \Omega$, we define 
\[
G_{a \shortto b}:=\{ g \in  G \ :\ g (a) = b \},
\]
these sets are called \textsl{the stars},
and we use $\ch_{a, b}$ for $\ch_{G_{a \shortto b}}$. 
We will use the notation $A(G,\Omega)$, or simply $A$, for the incidence matrix, in which the row are indexed by the elements of the group $G$, and the columns are the vectors $\ch_{a,b}$, $\forall a, b \in \Omega$.

Two elements $g, h \in G$ are \textsl{intersecting} if there is some $\alpha \in \Omega$ such that $g(\alpha) = h(\alpha)$. Further, a subset of $S \subset G$ is \textsl{intersecting} if any two elements in $S$ are intersecting. There are many recent results considering the size of the largest intersecting sets in different transitive permutation groups. In particular, the focus has been on finding groups in which a largest intersecting set has the same size as the stabilizer {$G_{\alpha \shortto \alpha}$} of a point $\alpha \in \Omega$---groups with this property are said to have the \textsl{Erd\H{o}s-Ko-Rado} property, or the EKR property.
Maximum intersecting sets in $2$-transitive groups have been well-studied.
For example, it has been shown~\cite{MSi2019, MST2016} that every 2-transitive group has the EKR property. 
For several classes of groups all the maximum intersecting sets have been characterized~\cite{ AM20152, AM2015, BMMKK2015, CK2003,  ellis2012setwise, ellis2011intersecting, MR4600198, MR4426440, LSP, LPSX2018, PSUEKR, MR4521814, MS2011,  Spiga2019, MR2419214}. In these groups, the cosets of a stabilizer of a point are intersecting sets of maximum size; these sets are called the \textsl{canonical intersecting sets}. In many groups, only the canonical intersecting sets are maximum intersecting sets, but in all 2-transitive groups the characteristic vector of an intersecting set of maximum size is a linear combination of characteristic vectors of the canonical intersecting sets~\cite{MSi2019}. This means, for any 2-transitive group, the characteristic vector of any intersecting set of maximal size is contained in $\mathrm{im}(A)$.  Similar to the projective case, we will define CL sets in groups as sets with a characteristic vector contained in $\mathrm{im}(A)$, and so, we will look for 01-vectors in this vector space. More precisely, we have the following definition. 

\begin{mydef} A \textsl{$(G,\Omega)$-Cameron-Liebler set}, or simply a CL set where the group and action are clear, is a subset of $G$, whose indicator function is a linear combination of elements in 
\[
\{\ch_{a, b}\ :\ a, b  \in \Omega\}.
\] 
\end{mydef}

For any pair $a, b \in \Omega$ the set $\ch_{a, b}$ is clearly a CL set, as is any disjoint union of sets of the form $G_{a \shortto b}$. The key question we consider here is if there are other CL sets.

\begin{mydef}
A CL set that is either a set $G_{\alpha \shortto \beta}$ or a disjoint union of sets of the form $G_{\alpha \shortto \beta}$ is called \textsl{canonical}. A \textsl{non-canonical} CL set is a CL set which is not canonical. A CL set which does not contain a set $G_{\alpha \shortto \beta}$ is called \textsl{proper non-canonical}. A \textsl{minimal} CL set, is a CL set with no CL set as a proper subset.
\end{mydef} 

\begin{prop}
If $C$ is a canonical CL set in a 2-transitive group $G$, then 
$v_C = \sum\limits_{a , b \in \Omega} c_{a, b} \ch_{G_{a \shortto b}}$
with either
$c_{a,b} = 0$ for all $a \neq a_0$, for some $a_0$; or $c_{a,b} = 0$ for all $b \neq b_0$, for some $b_0$.
\end{prop}
\begin{proof}
Since $C$ is a canonical CL set, it must be the union of disjoint canonical intersecting sets.
As the group $G$ is 2-transitive, two sets  $G_{a_1 \shortto b_1}$ and $G_{a_2 \shortto b_2}$ are disjoint only if $a_1 = a_2$ or $b_1 =b_2$.
\end{proof}

The first major result on CL sets for permutations was given in 2011 by Ellis~\cite{ellis2011intersecting}.

\begin{thm}[{\cite[Corollary 2]{ellis2011intersecting}}]\label{sym}
Consider the natural action of $\sym(n)$ on $[n]$. Point stabilizers and their cosets are the only minimal $(\sym(n), [n])$-CL sets and canonical CL sets are the only CL sets.
\end{thm}

Our focus is on permutation group that have non-canonical CL sets. The first class of permutation groups we consider are the class of Frobenius groups. We found that these groups have many proper non-canonical CL sets, and in fact were able to characterize all the CL sets in these groups in Section~\ref{sec:Frobenius}.

We define the parameter $x$ of a CL set $L$ by 
\[
x=\frac{|L|}{|G_{a\shortto b}|}=\frac{|L||\Omega|}{|G|}.
\]  
As a preliminary result (see Lemma~\ref{twodirections} and Corollary~\ref{cortussen01}), we show that the parameter of every CL set in a $2$-transitive group is at least one. Moreover, we also show that every CL set of a $2$-transitive group with $x=1$ is necessarily a maximum intersecting set (Corollary~\ref{cortussen01}). In this regard, it is interesting to construct non-canonical CL sets of parameter greater than one in 2-transitive groups. By an affine type $2$-transitive group, we mean a $2$-transitive group whose socle is an elementary abelian group. Applying the construction described in Theorem~\ref{switch} (also in Theorem~\ref{switchchar}), in Section~\ref{sec:CLinAffine}, we show that all affine type $2$-transitive groups have a non-canonical CL set of parameter greater than one. 


It is well-known (in fact, it is originally from Burnside's 1897 book~\cite[Theorem IX]{burnside}) that a $2$-transitive group which is not affine, must be almost simple. So we turn our attention to some almost simple groups. Since the symmetric group does not have any non-canonical CL sets, and we will see that the alternating group also has no non-canonical CL sets (Theorem~\ref{thm:NoAlternating}), we focus on the group $\psl(2,q)$ (where $q$ is a power of an odd prime) with its $2$-transitive action on the points of the projective line. The main result of~\cite{LPSX2018} states that every maximum intersecting set in $\psl(2,q)$ is necessarily canonical. Thus, by Corollary~\ref{cortussen01}, every parameter one CL set is canonical, but we are able to construct CL sets with larger parameters. Theorem~\ref{switch}, in Section~\ref{sec:CLinPSL}, gives a construction of a non-canonical CL set of parameter $(q-1)/{2}$. 

We consider some groups that are not $2$-transitive, here the CL sets are not as well behaved. For instance, in Example~\ref{fractional}, we demonstrate the existence of a CL set of parameter strictly less than $1$, in a permutation group which is not $2$-transitive. By Corollary~\ref{cortussen01}, such a result is not possible in the case of $2$-transitive groups. In Section~\ref{sec:CLsubgroup}, we focus on CL sets that are also subgroups. We will show that, by arguments from representation theory, it is easy to check whether a subgroup is a CL subgroup. We also give some examples of non-canonical CL subgroups.

In Section~\ref{sec:CLfromnondisj} we discuss CL sets coming from the union of non-disjoint stars. For example, in Subsection~\ref{subsec:psl(2,q^2)}, we construct another non-canonical CL set in $\psl(2,q^2)$. 
Finally, we end the paper with some open questions and future research.

\section{Tools from Algebra}\label{sec:introrep}

In this section we state some results from representation theory of groups that will be useful to determine when a set is a CL set. 

Indicator functions of subsets of $G$ are elements of the group algebra $\bb{C}[G]$ and the linear span of the indicator functions  
\[
\{\ch_{a, b}\ :\ a, b  \in \Omega\}
\] 
is a $2$-sided ideal of $\bb{C}[G]$, denoted by $\mf{I}_{G}(\Omega)$. It is well-known that the simple ideals of $\bb{C}[G]$ are indexed by the irreducible characters of $G$. In fact, all simple two-sided ideals of $\bb{C}[G]$ are of the form $\left \langle e_{\phi} \right\rangle$, for some $\phi \in Irr(G)$. 
For any two-sided ideal $\mathcal{J}$, there is a subset $J \subset Irr(G)$ such that $\mathcal{J}= \sum\limits_{\phi \in J} \langle e_{\phi} \rangle$. We will derive such a decomposition of $\mf{I}_{G}(\Omega)$.

Given any $\phi \in Irr(G)$, the primitive central simple idempotent
\[
e_{\phi}=\dfrac{ \phi (1)}{|G|} \sum\limits_{g \in G} \phi(g^{-1})g,
\]
generates a simple two-sided ideal of $\bb{C}[G]$; this ideal is denoted by $\langle e_{\phi} \rangle $. 
This ideal is identified with a vector $G$-space module with dimension $|G|$, the projection to this module is given by
\[
E_{\phi}(g,h)=\dfrac{ \phi (1)}{|G|} \phi(hg^{-1}).
\]

We use ${\one_G}$ to denote the trivial representation of $G$, and drop the subscript if the group is clear from context, its projection to the corresponding $G$-module is
\[
E_{\one_G}=\dfrac{ 1 }{|G|} J,
\]
where $J$ is the all ones matrix. 

For a subset $S \subseteq G$ and a complex character $\phi$, define $\phi(S):= \sum_{s \in S} \phi(s)$. 
We will use ${\bf 0}$ for the zero matrix. We can also consider indicator functions as length-$|G|$ vectors, rather than an element of the group algebra. The indicator vector of a set will have entry 1 in positions corresponding to elements in the set, and 0 elsewhere.

\begin{lemma}\label{tech}
Let $C \subset G$ and $\phi$ be a complex character of $G$ such that $\phi(C t^{-1})=0$ for all $t \in C$. Then $\phi(Cg)=0$ for all $g \in G$.
\end{lemma} 
\begin{proof}
Let $n = \phi(1)$ be the degree of $\phi$, and $\Phi : G \to GL_{n}(\bb{C})$ be a unitary representation affording $\phi$. (The existence of such a representation follows from a proof of Maschke's theorem.) 
As $\Phi$ is a unitary representation, for any $g \in G$, $\Phi(C^{-1}g^{-1})=\Phi(Cg)^{T}$. 
So we have 
 \begin{align*}
 \Phi(Cg) \times \Phi (Cg)^{T} &= \sum \limits_{t \in C}  \Phi (Ct^{-1}).
 \end{align*}
By our assumptions, $\phi(Ct^{-1})=Tr( \Phi (Ct^{-1}) )=0$ for all $t \in C$. 
We conclude for all $g \in G$, that $Tr( \Phi(Cg) \times \Phi(Cg)^{T} )=0$. 
Thus $\Phi(Cg)={\bf 0} $ and $\phi(Cg)=0$ for all $g \in G$. 
\end{proof}

We now describe the decomposition of  $\mf{I}_{G}(\Omega)$ as a direct sum of simple two-sided ideals of $\bb{C}[G]$.

\begin{lemma}[{\cite[Lemma 12]{Li-Pantangi}}]\label{idekrm} 
For a finite group $G$, with a subgroup $H$, consider the action of $G$ on $\Omega=G/H$.
Define $Y_{\Omega}:=\{\phi \in Irr(G)\ :\ \sum\limits_{h\in H} \phi (h)\neq 0\}$, then 
 \[
 \mf{I}_{G}(\Omega)=\oplus_{\phi \in Y_\Ome} \left \langle e_{\phi} \right \rangle.
 \]
\end{lemma}  

\begin{proof}
For any subset $S \subset G$ and $\phi \in Irr(G)$, we have
\begin{align*}
\frac{|G|}{\phi(1)}e_{\phi} \sum_{s \in S} s&= 
\frac{|G|}{\phi(1)}\sum\limits_{s \in S} e_{\phi}s  \notag
=\frac{|G|}{\phi(1)}\sum\limits_{s \in S} \frac{\phi(1)}{|G|} \sum\limits_{g \in G} \phi(g^{-1} )  g s  \notag
=\sum\limits_{s \in S}\sum\limits_{g \in G} \phi(sg^{-1})g \notag 
=\sum\limits_{g\in G} \phi( S g^{-1} )g.
\end{align*} 
Therefore for any $\phi \in Y_{\Omega}$, we have $e_{\phi} \left(\sum_{h \in H} h \right) \neq 0$
and $e_{\phi}  \sum_{h \in H} h  \in \left \langle e_{\phi}\right\rangle \cap \mf{I}_{G}(\Omega) \subset \mf{I}_{G}(\Omega)$. As $\left\langle e_{\phi}\right\rangle$ is a minimal ideal, we conclude that $\left \langle e_{\phi} \right \rangle \subset \mf{I}_{G}(\Omega)$.  

Next consider $\theta \in Irr(G) \setminus Y_{\Ome}$. In this case, since $\theta$ is a character, $\theta(H^g)=\theta(H)=0$ for any $g \in G$. By Lemma~\ref{tech}, we must have $\theta( H^{g}x )=0$ for all $x,g \in G$, and thus $e_{\theta}$ annihilates all elements of $Irr(G) \setminus Y_{\Omega}$. 
\end{proof}

As a corollary of the above, we obtain a character theoretic formulation 
\begin{cor}\label{CLform}
For $G$ be a finite group with $H \leq G$, consider the action of $G$ on $\Omega = G/H$ for some subgroup $H$. A subset $C$ is a $(G,\Omega)$-CL set if and only if  $\sum\limits_{x\in Ct^{-1}} \phi(x)=0$ for all $\phi \notin Y_{\Omega}$ and $t \in C$.
\end{cor}
\begin{proof}
Assume that $C$ is a subset such that $\phi (Ct^{-1})=0$ for all $\phi \notin Y_{\Omega}$ and $t \in C$. Showing that $C$ is a CL set is equivalent to showing that $\sum_{c \in C} c \in \mf{I}_{G}(\Ome)$. 

First, observe that 
\begin{equation}\label{1}
\frac{|G|}{\phi(1)} e_{\phi}\sum_{c \in C} c=\sum\limits_{g\in G} \phi({C g^{-1}})g.
\end{equation}
This means for a $\phi \notin Y_{\Omega}$, by Lemma~\ref{tech}, $e_{\phi} \sum_{c \in C} c =0$.  Using Lemma~\ref{idekrm}, orthogonality of idempotents $e_{\phi}$, and the fact that $\sum_{\phi \in Irr(G)}e_{\phi}=1$, we deduce that if $e_{\phi}\sum_{c \in C} c=0$ for all $\phi \notin Y_{\Omega}$, then $\sum_{c \in C} c\in \mf{I}_{G}(\Omega)$. Thus $\sum_{c \in C} c \in \mf{I}_{G}(\Omega)$.
 
The other direction of the statement follows from \eqref{1}. 
\end{proof}

These lead to a nice formulation for subgroups which are CL sets.
\begin{cor}\label{cor:CLsubform}
Let $G$ be a finite group, $H$ a subgroup, and consider the action of $G$ on $\Omega:=G/H$. A subgroup $C$ is a $(G,\Omega)$-CL set if and only if  $\sum\limits_{x\in C} \phi(x)=0$ for all $\phi \notin Y_{\Omega}$.
\end{cor}

In the case that the action of $G$ is 2-transitive on $\Omega$, the decomposition of $\mf{I}_{G}(\Omega)$ is very simple.
The set $Y_{\Omega}$ consists of only two representations, the trivial and the representation $\psi$ defined by $\psi(g) = \fix(g) - \one$. The representation $\one_G + \psi$ is equal to the representation induced on $G$ by $\one_{G_\alpha}$.
The character $ \one_G + \psi$ is known as the permutation character and the two-sided ideal corresponding to this character is the permutation module. In~\cite{AM20152}, it is shown that the canonical intersecting sets are a spanning set for the module.
If $G$ is 2-transitive, a set $C$ is a CL set if and only if 
\[
E_{\one_G + \psi}  (\ch_C) = \ch_C,
\]
or, equivalently, if $\ch_C$ is a linear combination of the indicator functions for the canonical intersecting sets.

A useful tool when considering intersecting sets in a group is the \textsl{derangement graph}. This is graph, denoted $\Gamma(G, \Omega)$, whose vertices are the elements of $G$ and two elements are adjacent if they are not intersecting, with the action of $G$ on $\Omega$. An intersecting set is a coclique, or independent set, in the derangement graph. The derangement graph is a normal Cayley graph \cite[Section 14.6]{GMbook}, specifically
\[
\Gamma( G, \Omega) = \Cay(G, \Der(G))
\]
where $\Der(G)$ is the set of derangements in $G$ (these are the elements with no fixed points). Vertices $g, h$ are adjacent if and only if $gh^{-1}\in \Der(G)$, and the degree of every vertex is the number of derangements in $G$,  which is denoted by $\der(G)$. For a fixed $g \in G$, let $N(g)$ be the set of all $h \in G$ for which $gh^{-1}$ is a derangement---this is the neighbourhood of $g$ in the derangement graph of $G$.

The next result shows how the cliques in the derangement graph can be used to get information about the CL sets; this result is not new, it can also be found in~\cite[Theorem 2.8.4]{PalmarinThesis} where is it stated only for 2-transitive groups.

\begin{lemma}\label{n-clique}
Let $G$ be a permutation group acting on a set $\Omega$, and assume the derangement graph $\Gamma(G, \Omega)$
contains a clique of size $|\Omega|$. Given a $(G,\Omega)$-CL set $L$, there is a natural number $n_{L}$ such that $|L \cap C|=n_{L}$, for any clique $C$ of size $|\Omega|$.    
\end{lemma}
\begin{proof}
Let $\ch_C$ be the characteristic vector of the clique $C$ with $|C|=|\Omega|$ and let $\onevec$ be the all one vector. Then $A^T \cdot \ch_C = \onevec$ and $A^T \cdot \onevec = \frac{|G|}{|\Omega|}\onevec $. Hence, $A^T \left(\ch_C-\frac{|\Omega|}{|G|} \onevec \right) = 0$. Now, if $L$ is a CL set, then, by definition, $\ch_L = A v$ for some vector $v$. Hence, 
\begin{align*}
    |L\cap C| -\frac{|\Omega|}{|G|}|L| = \ch_L^T\ch_C - \ch_L^T\frac{|\Omega|}{|G|} \onevec = \ch_L^T(\ch_C-\frac{|\Omega|}{|G|} \onevec )=v^T A^T (\ch_C-\frac{|\Omega|}{|G|} \onevec ) = 0.
\end{align*} This implies that $|L\cap C| =\frac{|\Omega|}{|G|}|L|$ for every clique $C$ of size $|\Omega|$.
\end{proof}

\begin{cor}\label{n-cliquesize}
Let $G$ be a permutation group acting on a set $\Omega$, such that the derangement graph $\Gamma (G, \Omega )$ contains a clique the same size of $\Omega$. Then the size of a $(G, \Omega)$-CL set is divisible by $\frac{|G|}{|\Omega|}$.
\end{cor}

\section{CL sets in Frobenius groups}\label{sec:Frobenius}

A Frobenius group is a finite permutation group in which no non-trivial element fixes more than one point. Let $G \leq \sym(\Omega)$ be a Frobenius group and $H:=G_{\omega}$ be the stabilizer of $\omega \in \Omega$. By a celebrated result by Frobenius (see~\cite{isaacs1994character} or any standard book on character theory), we know that
\[
K:= \left( G\setminus \bigcup \limits_{g \in G} gHg^{-1} \right) \cup \{e\}
\]
is a normal subgroup which is regular with respect to the action of $G$ on $\Omega$. The group $K$ is called the \textsl{Frobenius kernel} and the group $H$ is called the \textsl{Frobenius complement}. 

Applying the results in the previous section, we will find all $(G, \Omega)$-CL sets where $G$ is a Frobenius group.
First, we demonstrate some non-canonical CL sets in $G$. These are inspired by the structure of the derangement graph $\Gamma_{G}$ of the Frobenius groups $G$. By \cite[Theorem 3.6]{AM2015}, the cosets of $K$, partition $\Gamma(G, \Omega)$ into a disjoint union of cliques {of size $|\Omega|$}. In the following lemma, we construct minimal CL sets by taking exactly one vertex from each clique in this partition.

\begin{lemma}\label{minfrb}
Let $G$ be a Frobenius group with $K$ the Frobenius kernel and $H$ the Frobenius complement.
Let $f: H \to K$ be a function. Then the set $S_{f}:=\{f(h)h \ :\ h \in H\}$ is a $(G,\ \Omega)$-CL set.
\end{lemma}
\begin{proof}
To show that this is a CL set, we will use Corollary~\ref{CLform}, so we need to compute $\phi(S_{f}t^{-1})$, for all $\phi\in Irr(G)$ and $t\in S_{f}$.

Fix $h_{0} \in H$ and consider $t:=f(h_{0})h_{0} \in S_f$. As $K$ is normal, $y f(h)y^{-1} \in K$, for all $y,h \in H$. Define 
\[
f_{t}(y):= f(yh_{0}) \ y f(h_{0})^{-1}y^{-1},
\] 
for all $y\in H$. Now we observe that $S_{f}t^{-1}=S_{f_{t}}$. We have $f_{t}(e)e=1$ and for $y\neq e$, we have 
$f_{t}(y)y \in Ky$. By a well-known (see Problem 7.1 of \cite{isaacs1994character}) property of Frobenius group, given $h \in G \setminus K$, all the elements of the coset $Kh$ are conjugate to $h$. Therefore, we have 

$$\sum\limits_{y\in H} \phi(f_{t}(y)y)=\sum\limits_{y\in H} \phi(y)=\phi(H).$$

Now the result follows from Corollary~\ref{CLform}.

\end{proof}

We are now ready to give a characterization of the CL sets in a Frobenius group, to do this, we use Lemma~\ref{n-clique}. 

\begin{thm}\label{thm:FrbCLsets}
Let $G \leq \sym(\Omega)$ be a Frobenius group, let $K \triangleleft G$ be its Frobenius kernel and $H$ its Frobenius complement. If $L$ is a minimal $(G,\Omega)$-CL set, then there is a function $f: H \to K$ such that $L= \{f(h)h\ : \ h \in H\}$.
\end{thm}
\begin{proof} Let $L$ be any $(G,\Omega)$-CL set. Observing that any coset of $K$ is a clique of size $|\Omega|$, by Lemma~\ref{n-clique}, we have $|L\cap Kh|=|L \cap K|$ for all $h \in H$. Set $k:=|L \cap K|$, since $L$ is minimal $k=1$. As $G= \cup_{h \in H} Kh$, for each $h$ we set by $L \cap Kh = \{k_h h  \}$, then we can define $k$ functions $\left(f : H \to K\right)$ by $f(h) = k_h$. With this definition it is clear that $L=\cup S_f$. By Lemma~\ref{minfrb}, we know $S_f$ is a CL sets. Since the union of any disjoint minimal CL sets is also a CL set, this proves the theorem. 
\end{proof}

\section{CL sets of 2-transitive groups}
\label{sec:2transenswitch}

In this section, we prove that a $(G,\Omega)$-CL set, where the action of $G$ on $\Omega$ is 2-transitive, corresponds to an equitable partition of the derangement graph $\Gamma(G, \Omega)$ with two parts (these are also known as \textsl{intriguing sets}~\cite{MR2679936}). As a corollary, we find that the size of a non-empty CL set in $G$ is at least the size of a canonical example $G_{a \shortto b}$. 

Recall that $A(G, \Omega)$ is the incidence matrix for the action---the rows are indexed by the group elements of $G$, and the columns are indexed by the pairs $(a,b)\in \Omega \times \Omega$; then $A_{g, (a, b)} = 1$ if $g(a) = b$ and 0 otherwise. We use $A$ where the group and action are clear. Each column of $A(G, \Omega)$ is the characteristic vector of a canonical intersecting set. A set $L$ is a CL set if and only if $\ch_L \in \mathrm{im}(A) = \ker(A^T)^\perp$, (where $\perp$ is taken with respect to the standard inner product on columns.) The \textsl{incidence vector} for an element $g \in G$ is the row corresponding to $g$ in $A(G, \Omega)$---this is a 01-vector with the indexed by $\Omega \times \Omega$ with the $(a,b)$-entry equal to 1 if and only if $g (a) = b$. This vector is denoted by $v_g$. We use $N(g)$ to denote the neighbourhood of $g \in G$ in the the derangement graph of $G$; equivalently, this is the set
\[
N(g) = \{ h \in G \ : \ hg^{-1} \textrm{ is a derangement}  \}.
\]

\begin{lemma}\label{lemmaaantaldisjunct} 
Let $G$ be a group with a $2$-transitive action on the set $\Omega$, with $|\Omega| = n$, and $A$ be its incidence matrix.
Then for any $g \in G$ and $a, b \in \Omega$,
\begin{align*}
\ch_{N(g)} -\frac{\der(G)}{n-1}\left(\frac{n}{| G |}  \onevec  -\ch_g \right) \in \ker(A^T).
\end{align*}
\end{lemma}
\begin{proof}
First note $A^T \onevec=  \frac{|G|}{n} \onevec $, and if $v_g$ is the incidence vector of the one element set $\{g\} \subset G$, then $A^T \ch_g = v_g$. 

The $(a, b)$-entry of $A^T \ch_{N(g)}$ is the number of $h \in G$, adjacent in $g$ in the derangement graph, such that $h(a) = b$. If $g(a)=b$, then there are no such elements adjacent to $g$. Provided that $g(a) \neq b$, since $G$ is $2$-transitive, there are $\frac{\der(G)}{n-1}$ elements $h\in G$ with $h(a)=b$ and $gh^{-1}\in \Der(G)$. This implies that
\[
A^T \ch_{N(g)} = \frac{\der(G)}{n-1}( \onevec - v_g)
= \frac{\der(G)}{n-1} \left( \frac{n}{|G| } 
   A^T \onevec - A^T \ch_g \right).
\] 
This is equivalent to
\[
 \ch_{N(g)} - \frac{\der(G)}{n-1} \left( \frac{n}{|G|}  \onevec -\ch_g \right) \in \ker(A^T)
\]
and the lemma now follows.
\end{proof}

In the next lemma, we use the notation $\delta_{L}(g)$ for the indicator function; this function equals one if $g \in L$ and is zero otherwise.

\begin{lemma}\label{twodirections}
Let $G$ be a group acting $2$-transitively on the set $\Omega$ where $|\Omega|=n$.
A set $L$ in $G$ is a $(G, \Omega)$-CL set with parameter $x$ if and only if for every $g\in G$, the number of group elements $h \in L$ with $gh^{-1}$ a derangement is exactly $(x-\delta_{L}(g)) \frac{\der(G)}{n-1}$.
\end{lemma}

\begin{proof}
Assume $L$ is a $(G, \Omega)$-CL set with parameter $x$, then $\ch_L \in \mathrm{im}(A)= \ker(A^T)^\perp$. 
By Lemma~\ref{lemmaaantaldisjunct}, we know that 
\begin{align*}
&\ch_{N(g)} -\frac{der(G)}{n-1} \left(\frac{n}{|G| } \onevec -\ch_g \right) \in \ker(A^T).
\end{align*}
Since $\ch_L \in \ker(A^T)^\perp$, this implies
\[
0 =  \ch_{N(g)} \cdot \ch_L  -
         \frac{der(G)}{n-1} \left( \frac{n}{|G|} \onevec \cdot \ch_L  - \ch_g \cdot \ch_L  \right)=
 | N(g) \cap L | -\frac{der(G)}{n-1} \left( \frac{n}{|G|} | L | - \delta_{L}(g)  \right), 
\]
which is equivalent to
\[ 
\ | N(g) \cap  L | =(x-\delta_{L}(g) ) \frac{der(G)}{n-1}.
\]
This proves the first direction.

Suppose now that $L$ is a set such that for every $g \in G$, the number of group elements $h$ of $L$, with $gh^{-1}\in \Der(G)$ is exactly $( x- \delta_{L}(g)  )\frac{der(G)}{n-1}$.
Let $M(G)$ be the adjacency matrix for the derangement graph of $G$, then this implies that
\[
M(G) \ch_L = \frac{der(G)}{n-1}(x \onevec -\ch_L ), \qquad M(G) \onevec = der(G) \onevec.
\] 
Let $v= \ch_L - \frac{x}{n} \onevec $. Then we have that 
\begin{align*}
    M(G)  v & = M(G) (\ch_L - \frac{x}{n} \onevec )  
                = \frac{der(G) \ x}{n-1} \onevec - \frac{der(G) }{n-1} \ch_L - \frac{der(G) \ x}{n} \onevec \\
               & = -\frac{der(G)}{n-1} \left( \ch_L  + x(1-\frac{n-1}{n} ) \onevec \right) 
                = -\frac{der(G)}{n-1} v. 
\end{align*}
This implies that $v$ is an eigenvector for $M(G)$ 
with eigenvalue $-\frac{der(G)}{n-1}$. Since $G$ acts 2-transitively, and 
from~\cite[Lemma 4.1.]{AM20152} we know that $\ch_L = v - \frac{x}{n} \onevec $ is in the permutation module, and hence, $L$ is a $(G, \Omega)$-CL set with parameter $x$.
\end{proof}

Consider a 2-transitive group $G$ in which $L$ is a CL set with parameter $x$. Then $L$ corresponds to a set of vertices in the derangement graph $\Gamma$. By the previous lemma, for every vertex $g \in \Gamma$, the number of vertices of $L$ adjacent to $g$ is equal to 
$\frac{ \der(G) }{n-1} ( x-\delta_{L}(x) )$. This number only depends on whether $g$ is in $L$ or not, which implies the set $L$ is an equitable partition in the derangement graph. The next result shows that this also gives a lower bound on the size of a CL set in a 2-transitive group, a special case of this result is given in~\cite[Theorem 2.8.4]{PalmarinThesis}.

\begin{cor} \label{cortussen01}
Let $G$ be a group with a $2$-transitive action on $\Omega$, and let $L$ be a non-empty $(G,\Omega)$-CL set with parameter $x$. Then $x\geq 1$, further, if $x=1$, then $L$ is a maximum intersecting set.
\end{cor}
\begin{proof}
Suppose there is a $(G, \Omega)$-CL set $L$ with parameter $x \in (0,1)$. The set $L$ is not empty, so suppose $g \in L$. By Lemma~\ref{twodirections}, if $x<1$, then the number of group elements $h \in L$ with $gh^{-1}\in \Der(G)$ is negative, which gives the contradiction. If $x=1$, then the number of group elements $h \in L$ with $gh^{-1}\in \Der(G)$ is zero for every $g \in L$, hence every two elements of $L$ intersect. Due to the size of $L$, it follows that $L$ is a maximum intersecting set.
\end{proof}

Next we will describe a method that can be used to build a non-canonical CL set from a canonical CL set for some groups. In the context of CL sets in projective spaces, this method was introduced by Cossidente and Pavese in \cite{CosPavReplacement} under the name \textsl{replacement technique}.
For a $g \in G$, recall the character $\psi(g) = \fix(g)-\one$. Since $G$ is 2-transitive, $\psi$ is an irreducible representation of $G$, and $\one + \psi$ is the permutation representation. For a set $X \subset G$, we will use the notation $gX = \{ gx \ : \ x \in X \}$ for any $g \in G$ and note that $\psi(gX) = \psi(v_{gX})$.

\begin{thm}\label{switch}
Let $G$ be a group with a $2$-transitive action on $\Omega$, where $|\Omega|=n$, and let $d= \der(G)$. Assume $L$ is a $(G,\Omega)$-CL set with parameter $x$ and let $X, Y$ be sets of group elements of $G$ of equal size such that:
\begin{enumerate}
    \item $X \subset L$ and $| Y \cap L | = 0$; \label{prop1}
    \item if $g \notin X \cup Y$, then $|N(g) \cap X| = |N(g) \cap Y|$; \label{prop2}
    \item if $g \in X$, then $ |N(g) \cap X| - |N(g) \cap Y| =\frac{-d}{n-1}$;   \label{prop3}
    \item if $g \in Y$, then $|N(g) \cap X| - |N(g) \cap Y| =\frac{d}{n-1}$.  \label{prop4}
\end{enumerate}
Then the set $\Bar{L}= (L \setminus X)\cup Y$ is a $(G, \Omega)$-CL set with parameter $x$.
\end{thm}
\begin{proof}
Since $L$ is a $(G, \Omega)$-CL set with parameter $x$, we have
that 
\[
| \{ h \in L \ | \ gh^{-1} \in \Der(G) \} | =  (x- \delta_{L}(g)) )  \frac{d}{n-1}
\]
by Lemma~\ref{twodirections}.

For $g \in G$, we consider the following cases
\begin{enumerate}
    \item If $g \in L \setminus X$, then $g \in \Bar{L}$. From Condition~\ref{prop2}
    \[
    |\{h \in \Bar{L} \ | \ gh^{-1}\in \Der(G)\}|=  \frac{d  (x-1) }{n-1}.
    \]
    \item If $g \notin L \cup Y$, then $g \notin \Bar{L}$. From Condition~\ref{prop2},
    \[
    |\{h \in \Bar{L}\ | \ gh^{-1}\in \Der(G)\}| = \frac{dx}{n-1}.
    \]
    \item If $g \in  X$, then $g \notin \Bar{L}$. From Condition~\ref{prop3}, 
      \begin{align*}
            |\{h \in \Bar{L} \ | \ gh^{-1}\in \Der(G)\}|
            &= 
            |\{h \in L \ | \ gh^{-1}\in \Der(G)\}| - |N(g) \cap X | + |N(G) \cap Y | \\
            & =
            \frac{d (x-1) }{n-1} - \frac{-d}{n-1} \\
            &= \frac{dx}{n-1}.
    \end{align*}      
    \item If $g \in  Y$, then $g \in \Bar{L}$. From Condition~\ref{prop4}, 
      \begin{align*}
            |\{h \in \Bar{L} \ | \ gh^{-1}\in \Der(G)\}|
            &= 
            |\{h \in L \ | \ gh^{-1}\in \Der(G)\}| - |N(g) \cap X | + |N(G) \cap Y | \\
            & =
            \frac{d x }{n-1} - \frac{d}{n-1} \\
            &= \frac{d(x-1)}{n-1}.
     \end{align*}    

\end{enumerate}
The proof follows from Lemma~\ref{twodirections}.
\end{proof}

The theorem also has the following character theoretic version. 

\begin{thm}\label{switchchar}
Let $G$ be a group acting $2$-transitively on $\Omega$, and let $\psi \in Irr(G)$ be such that $1+\psi$ is the permutation character. Let $L$ be a $(G,\Omega)$-CL set with parameter $x$ and $X \subset L$, and $Y \subset G \setminus L$ be sets of group elements of $G$ of equal size. Then $\bar{L}=\left( L \setminus X \right) \cup Y$ is a CL set of parameter $x$ if and only if
\[
\psi ( X g^{-1}  ) - \psi ( Y g^{-1} ) = ( \delta_{X}(g) - \delta_{Y}(g) ) \frac{|G|}{\psi (1)},
\]
for all $g \in G$.
\end{thm}
\begin{proof}
Let $S$ be any subset of $G$.
By Lemma~\ref{idekrm}, we see that $S$ is a CL set if and only if
\[
(e_{1}+e_{\psi}) (\ch_{S}) = \ch_{S}.
\]
A quick computation shows that 
\[
( e_{1} + e_{\psi} ) (\ch_{S}) = \sum\limits_{g \in G}  \left(  \frac{| S |}{|G|} + \frac{ \psi(1) \psi ( \ch_{ S g^{-1}} ) } {|G|}   \right) g.
\]
It now follows that $S$  is a CL set if and only if 
\begin{equation}\label{permcharcl}
\frac{| S |}{|G|}+ \frac{ \psi(1) \psi (\ch_{ S  g^{-1}} ) }{ |G| } = \delta_{S}(g).
\end{equation}
For any $g \in G$, we have $\psi( \ch_{\bar{ L}g^{-1}} )=\psi( \ch_{L g^{-1}} )-\psi(\ch_{{X}g^{-1}})+\psi(\ch_{{Y}g^{-1}})$. The result now follows by applying \eqref{permcharcl} to $L$ and $\bar{L}$.
\end{proof}

\section{CL sets in affine type 2-transitive groups}\label{sec:CLinAffine}

In this section, we demonstrate the existence of non-canonical CL sets in all affine type $2$-transitive groups. By a well-known result (for example, see \cite[Theorem 4.7A]{dixon1996permutation}), the socle of a 2-transitive group is either a vector space over a finite field or a non-abelian simple group. By an affine type 2-transitive group, we mean a 2-transitive group whose socle is a vector space over a finite field. 

Throughout this section $V$ is a vector space of dimension $n\geq 2$ over $F$, where $F$ is a finite field of order $q$ (so $q$ is a prime power). For any such vector space $V$, it is possible to build a $2$-transitive group of affine type. Pick a subgroup $G_{0}<\mrm{GL}(V)$, which acts transitively on the set of non-zero vectors in $V$. Then the action of $G :=V \rtimes G_{0}$ 
on $V$, described by $(v,g)\cdot w = v+g(w)$ is $2$-transitive, that is, $G$ is an affine type 2-transitive group. It is well-known (for example see \cite[Theorem 4.7A]{dixon1996permutation}) that every affine type $2$-transitive group can be constructed in this way. We note that $G_{0}$ is the stabilizer of the zero vector in $V$.

We will describe the construction of non-canonical CL sets with parameter $q^{n-1}$ in any group $G :=V \rtimes G_{0}$. Consider a hyperplane $W<V$. Define $S :=\{ g \in G_{0}\ :\ g(W)=W\}.$ Then the subgroup $H:=WS \cong W\rtimes S$ is the stabilizer of $W$ with respect to the action of $G$ on hyperplanes. We also consider the subgroup $K=VS\cong V \rtimes S$. By $L$, denote the canonical CL set $\{g \in G :\ g(0) \in W\}$. We first note that $H \subset L$. Pick any $k \in K \setminus H$, then $k(0)\notin W$, and thus $kH \cap L = \emptyset$. 

\begin{thm}\label{thm:AffineExamples}
For any $k \in K \setminus H$, the set $M=(L \setminus H) \cup kH$ is a non-canonical CL set of parameter $q^{n-1}$.
\end{thm}

We will prove this theorem over the next three lemmas, but we set some notation first. In Lemma~\ref{lem:Hsize}--\ref{lem:mfreecosetcharsum}, $G$ and $H$ are the groups defined in the previous paragraph.
Let $\psi \in Irr(G)$ be such that $\one+\psi$ is the permutation character of $G$.  
By Theorem~\ref{switchchar}, proving that $M$ is a CL set is equivalent to showing that, for all $g \in G$, we have
\begin{equation}\label{eq:chdiff}
\psi(g^{-1}H)- \psi(g^{-1}kH) = \left(\delta_{H}(g)-\delta_{kH}(g)\right)\frac{|G|}{q^n-1}.
\end{equation}  
We denote the $q$-binomial coefficient by $\qbin{n}{r}_{q}$, and recall that
\[
\qbin{n}{r}_{q} ={\frac {(q^{n}-1)(q^{n-1}-1)\cdots (q^{n-r+1}-1)}{(q^r-1)(q^{r-1}-1)\cdots (q-1)}};
\]
this equation counts the number of $r$-dimensional subspaces in an $n$-dimensional vectors space over a finite field of over $q$.

\begin{lemma}\label{lem:Hsize}
$|H| = \frac{|G|}{q\qbin{n}{1}_{q}}$.
\end{lemma}
\begin{proof}
    As $G$ is $2$-transitive, $G_{0}\leq \mrm{GL}(V)$ acts
    transitively on the set of non-zero vectors of $V$. Consider the natural action of $G_{0}$ on the dual $V^{*}$ of $V$. Given $g \in G$, let $R_{g}$ be the matrix representation of $g$ with respect to a basis $\mathcal{B}$, and let $R^{*}_{g}$ be the matrix representation of the action of $g$ on $V^{*}$ with respect to a dual basis $\mathcal{B}^{*}$ of $\mathcal{B}$. As $R^{*}_{g}=R_{g^{-1}}^{T}$, both $R_{g^{-1}}-I$ and $R^{*}_{g}-I$ have the same rank. Therefore, the number of vectors in $V^{*}$ fixed by $g$, is the same as the number of vectors in $V$ fixed by $g^{-1}$. By Burnside's orbit-counting lemma, the action of $G_{0}$ on $V$ has the same number of orbits as the action of $G_{0}$ on $V^{*}$ and we can conclude that $G_{0}$ acts transitively on the set of hyperplanes in $V$. 
    
    Now by the orbit-stabilizer formula, we can conclude that $|S|= |G_{0}|  / {\qbin{n}{1}_{q}}$, and thus 
    \[
    \frac{|G|}{|H|} = \frac{q^n|G_0|}{q^{n-1}|S|}=q \qbin{n}{1}_{q}.
    \]
\end{proof}

\begin{lemma}
$\left\langle \psi,\ \ind^G(\one_{H}) \right\rangle =1$.
\end{lemma}
\begin{proof}
First, set $m=\left\langle \psi,\ \ind^G(\one_{H}) \right\rangle$. Note that since $H$ stabilizes $W$, the action of $H$ on $V$ has at least $2$ orbits. Therefore, by Burnside's orbit-counting lemma and Frobenius reciprocity, we have
\begin{align*}
2 \leq \frac{1}{|H|} \sum\limits_{h \in H} \left( 1 + \psi(h) \right) =1 + \left\langle \psi|_{H},\ \one_{H} \right\rangle_H 
 =1+ \left\langle \psi,\ \ind^G(\one_{H}) \right\rangle_G. 
\end{align*}

If $q=2$, then the orbits of $W$ on $V$, are exactly the $2$ cosets of $W$ in $V$, so by the orbit-counting lemma $m=1$. 
For any value of $q$, the degree of $\psi$ is $q^n-1$, and $\ind^G(\one_{H})$ has degree $|G|/|H|$. By the above inequality, $m (q^{n} -1) \leq |G|/|H|$ and with Lemma~\ref{lem:Hsize} we have
\[
1\leq m \leq \frac{|G|}{|H|(q^{n}-1)} = \frac{|G| q \qbin{n}{1}_{q} }{|G|(q^{n}-1)} \leq \frac{q}{q-1}.
\]
Provided that $q>2$, this implies $\frac{q}{q-1}<2$, and $m=1$.

\end{proof}

Since $\left\langle \psi,\ \ind^G(\one_{H}) \right\rangle =1$, we can compute $\psi(gH)$ quite easily using~\cite[Corollary 4.2]{isaacs2008character}. 

\begin{lemma}\label{lem:mfreecosetcharsum}
For all $g \in G$, we have 
\[
\psi(gH)  = \frac{|V||g(W) \cap W|-|W|^{2}}{|V|| W|-|W|^{2}} |H| .
\]
\end{lemma} 
\begin{proof}
If $\Psi : G \to \mrm{GL}(q^{n}-1,\ \bb{C})$ is a representation affording $\psi$ as its character, since $\left\langle \psi,\ \ind^G(\one_{H}) \right\rangle=1$, by~\cite[Corollary 3.2]{isaacs2008character}, we have $\psi(gH)= \Psi(g)_{1,1} \times |H|$. 

Consider $V$ as a $G$-module and denote the corresponding permutation module by $X := \bb{C}V$. If $\mathbf{1}$ is the all ones vector in $X$, then it is well known that $U := \mathbf{1}^{\perp}$ (here $\perp$ is with respect to the standard norm in $X$), is an irreducible $G$-representation (this is equivalent to $G$ being 2-transitive on $V$). Then $\psi$ is the character afforded by $U$. The vector $\mathbf{w}:=\mathbf{1}-\frac{|V|}{|W|} \ch_{W}$ is perpendicular to $\mathbf{1}$, so $\mathbf{w} \in U$. 

Let $\mathcal{B}$ be an ordered orthogonal basis for $U$ containing $\mathbf{w}$. If $\Psi: G \to \mrm{GL}(q^{n}-1,\bb{C})$ is the matrix representation of $U$ with respect to this basis, then 
\[
\Psi(g)_{1,1}= \dfrac{g(\mathbf{w})^{T}\mathbf{w}} {\mathbf{w}^T \mathbf{w}}.
\]
As $\mathbf{w}$ is orthogonal to $\mathbf{1}$, we have 
\[
g(\mathbf{w})^{T} \mathbf{w}
= g(\mathbf{w})^T  (-\frac{|V|}{|W|} \ch_{W})= (1 -\frac{|V|}{|W|} \ch_{g(W)} )^T (-\frac{|V|}{|W|} \ch_{W}) = -|V| + \frac{|V|^2}{|W|^2}|g(W)\cap W|.
\]

The result follows by using $\psi(gH)= \Psi(g)_{1,1} \times |H|$.
\end{proof}

Using the above Lemma, we can now prove Theorem~\ref{thm:AffineExamples} by simply verifying \eqref{eq:chdiff}. We now observe the following:
\begin{enumerate}
\item If $g \in H$, then $g(W) \cap W =W$, and $\psi(gH)= |H|$.
\item If $g \in K \setminus H$, then $g(W)$ is a non-trivial coset of $W$ in $V$, and thus $g(W) \cap W= \emptyset$ in this case. Thus $\psi(gH)= -\frac{1}{q-1}|H|$.
\item If $g \in G \setminus K$, then $g(W)$ is an affine hyperplane which is not a coset of $W$ in $V$. So $g(W) \cap W$ is an affine subspace of $V$ with dimension $n-2$, and thus by Lemma~\ref{lem:mfreecosetcharsum}, $\psi(gH)=0$.  
\end{enumerate}

As $|H|= \frac{|G|}{q^{n}-1} \times \frac{q-1}{q}$, \eqref{eq:chdiff} follows from the above observations. Therefore $M$ is a CL set. We now show that $M$ is non-canonical.

\begin{lemma}
   $M$ is a non-canonical CL set. 
\end{lemma}
\begin{proof}
 To show $M$ is non-canonical, it suffices to show that for all $x,y \in V$, we have $G_{x \shortto y} \not\subseteq M$. As $G$ is $2$-transitive, given $x_{1}, x_{2}, y_{1}, y_{2} \in V$ with $x_1\neq x_2$ and $y_1 \neq y_2$, by the orbit-stabilizer formula, we have 
 \begin{equation*}
 |G_{x_{1}\shortto y_{1}} \cap G_{x_{2}\shortto y_{2}}|= \dfrac{|G_{0}|}{|V|-1}. 
 \end{equation*}
Further, since $L=\bigcup\limits_{w \in W}G_{0 \to w}$, by the above equation we have for any $x \neq 0$
 \begin{equation}\label{eq:Lint}
     |G_{x\shortto y}\cap L|= \left( |W|-\delta_{W}(y) \right)  \dfrac{|G_{0}|}{|V|-1}.
 \end{equation}

The action of $H$ on $V$ has two orbits, $W$ and $V\setminus W$. Again, by the orbit-stabilizer formula, we have

\begin{equation}\label{eq:Hint}
     |G_{x\shortto y}\cap H|= \begin{cases} 
\dfrac{|H|}{|W|}\ &\text{if $x,y\in W$};\\
\dfrac{|H|}{|V|-|W|}\ &\text{if $x,y\in V \setminus W$; and }\\
0\ &\text{otherwise.}
     \end{cases} 
 \end{equation}

For any $k \in G$, we have $|G_{x\shortto y}\cap kH|=|k^{-1}G_{x\shortto y} \cap H|=|G_{x \shortto k^{-1}(y)} \cap H|$. If $k(0)=a$, by replacing $k$ with an appropriate element in $kH$, we may assume that $k=a \in V\setminus W \subset G$. Thus, we have 
\begin{equation}\label{eq:kHint}
    |G_{x\shortto y}\cap kH| = |G_{x \shortto y-a}\cap H|.
\end{equation}
Using $|H|=\dfrac{|W|(q-1)|G_{0}|}{|V|-1}$ (as $|V|=q|W|=q^n$) along with \eqref{eq:Lint}, \eqref{eq:Hint}, \eqref{eq:kHint}, we can conclude that \[|G_{x \shortto y}\cap M|< |G_{x \shortto y}|,\] for all $x,y \in V$. This concludes the proof.
\end{proof}

\begin{remark}
 Let $G$ be an affine type $2$-transitive group, whose socle $V$ is an $n$ dimensional vector space over $\bb{F}_{q}$, where $q=p^f$ for $f\in \bb{N}$ and $p$ a prime. Provided $nf>1$, by treating $V$ as an $nf$ dimensional space over $\bb{F}_{p}$, we can construct CL sets of parameter $p^{nf-1}$. In the case when $n=f=1$, $G$ must be the Frobenius group $\bb{F}_{p}\rtimes \bb{F}_{p}^{\times}$ and we know the structure of all CL sets in Frobenius groups. In conclusion, all $2$-transitive groups of affine type have non-canonical CL sets.
\end{remark}

\begin{remark}
    We want to note that it is also possible to prove that $M$ is a CL set, by using matrices. More precisely, $\ch_M$ can be written as a linear combination of characteristic vectors of stars:
    \begin{align*}
        q^{n-1}\ch_{M}= (q^{n-1}-1)\ch_{0\shortto W}+\ch_{0\shortto kW}-\ch_{W\setminus \{0\} \shortto W}+\ch_{W\setminus \{0\} \shortto kW}.
    \end{align*}
\end{remark}

\section{Lifting CL sets}\label{sec:lifting}

For a subgroup $H \leq G$, the restriction of $\ch_{G_{\alpha,\beta}}$ to $H$ is simply the vector $\ch_{H_{\alpha,\beta}}$. Further, for any CL set $L$ in $G$, if its characteristic vector is
\[
\ch_L = \sum_{\alpha,\beta} a_{\alpha,\beta} \ch_{G_{\alpha,\beta}},
\]
then the restriction of $L$ to $H$, is 
\[
\ch_{L \cap H } = \sum_{\alpha,\beta} a_{\alpha,\beta} \ch_{H_{\alpha,\beta}}.
\]
So $L \cap H$ is a CL set of $H$. Clearly, if $L$ is canonical, then $L \cap H$ is also canonical, but if $L$ is non-canonical, then $L \cap H$ may be either canonical or non-canonical. If $L$ is non-canonical, the coefficients $a_{\alpha,\beta}$ are not either 0 or 1, but it could still happen that $\sum_{\alpha,\beta} a_{\alpha,\beta} \ch_{H_{\alpha,\beta}}$ is canonical, this is because the vectors $\ch_{H_{\alpha,\beta}}$ are not linearly independent.

In this section we consider subgroups $H \leq G$ with index two and we will show that there is a weak converse to restriction. We will show in some cases we can ``lift" a CL set in $H$ to a CL set in $G$. 

The vector $\ch_{G_{\alpha,\beta}}$ can also be considered as a Boolean function by the map $\ch_{G_{\alpha,\beta}} : G \rightarrow \{0,1\}$, which is defined by $\ch_{\alpha,\beta}(g) = 1$ if $g(\alpha) = \beta$, and 0 otherwise. 
A function on $G$ is said to have \textsl{degree-$t$} if it can be expressed as a polynomial of degree $t$ in the functions $\ch_{G_{\alpha,\beta}}$. Following a method used by Filmus and Lindzey, we will show that any degree-4 function is orthogonal to the set of representations of $G$. For any 4-set $\Lambda$ in $\binom{\Omega}{4}$, let $G_\Lambda$ be the point-wise stabilizer of $\Lambda$. Define
\begin{align*}
    \Phi = \left\{ \phi \ : \ \langle \phi, \ind^G(\one_{G_\Lambda})\rangle = 0 \textrm{ for all }  \text{ordered tuples $\Lambda$ of four distinct elements} \right\}.
\end{align*}

\begin{lemma}\label{lem:degree4}
Let $G$ be a group acting on a set $\Omega$ and $\Phi$ defined as above. Then any degree-four function $f$, has $\langle f, \phi \rangle = 0$, for any $\phi \in \Phi$.
\end{lemma}
\begin{proof}
     For any $4$-set $\Lambda$, by definition of $\Phi$ and by Frobenius reciprocity, for all $\phi \in \Phi$,  we have $\left\langle\ch_{G_{\Lambda}}, \phi \right\rangle = \left\langle \ind^G(\one_{G_{\Lambda}}), \phi \right\rangle =0$. By Lemma~\ref{tech}, for any coset $gG_{\Lambda}$ of $G_{\Lambda}$, we have $\left\langle\ch_{gG_{\Lambda}}, \phi \right\rangle=0$. Given a subset $\Delta \subset \Lambda$, as $G_{\Delta}$ is a disjoint union of cosets of $G_{\Lambda}$, we have $\left\langle\ch_{G_{\Delta}}, \phi \right\rangle=0$. Again by Lemma~\ref{tech}, if $g$ is an indicator function of a coset of a point-wise stabilizer of $k$-set, with $k\leq 4$, we must have $\langle g, \phi \rangle=0$.  We now observe that a monomial of degree $4$ in vectors $\ch_{G_{\alpha,\ \beta}}$ is an indicator function of a coset of a point-wise stabilizer of some $k$-set with $k\leq 4$. We can now conclude that the lemma is true.
\end{proof}

\begin{thm}\label{thm:lift}
Let $H \leq G$ be groups with $[G:H]=2$.
If $\ind^G(\one_H)- \one \in \Phi$ (defined above) and $L$ is a CL set for $H$, then there exists a CL set in $G$, which contains $L$.
\end{thm}
\begin{proof}
Let $L$ be a CL set in $H$. Then there exist scalars $a_{i,j}$ such that 
\[
\ch_L = \sum_{\alpha,\beta} a_{\alpha,\beta} \ch_{H_{\alpha,\beta}}, 
\]
further, $\ch_L$ can be considered as a function on $H$. This function can be extended to a function on $G$ by
\[
g = \sum_{\alpha,\beta} a_{\alpha,\beta} \ch_{G_{\alpha,\beta}}. 
\]
By definition $g$ is degree one function. To prove that it is a CL set, we need to prove that it is Boolean. Define $f = g^2 (g-1)^2$; this is a positive function that is zero exactly if $g$ is Boolean. Since $f$ restricted to $H$ is Boolean, the value of $f$ restricted to $H$ is 0.

If $f \neq 0$, then, since $f$ is positive, the sum of $f$, so $\langle f, 1 \rangle$, will be strictly positive.
Since the restriction of $f$ to $H$ is 0, this means that $\langle f, \ind^G(\one_H) - \one \rangle > 0$, (note that $\ind^G(\one_H) - \one$ is the alternating character defined on $H$). So, $f$ is Boolean if and only if $\langle f,  \ind^G(\one_H) - \one \rangle = 0$. Since $g$ is degree 1, $f$ is a degree 4 function. Since $\ind^G(\one_H) - \one  \in \Phi$, Lemma~\ref{lem:degree4} implies that $\langle  f,  \ind^G (\one_H)- \one  \rangle = 0$, so $f$ is Boolean.
\end{proof}

Note that $\ind^G(\one_H) - \one \in \Phi$ occurs if and only if 
\[
0= \langle  \ind^G(\one_H) - \one ,\ \ind^G (\one_{G_\Lambda})\rangle
 = \langle  \res_{G_\Lambda} ( \ind^G (\one_H)- \one), \ \one \rangle_{G_\Lambda}
 = \frac{ 2 |H \cap G_\Lambda | - |G_\Lambda | }{|G_\Lambda |} \one.
\]
This occurs exactly if exactly half the elements of $G_\Lambda$ are in $H$ for every 4-set $\Lambda$ from $\Omega$.

\begin{thm}\label{thm:NoAlternating}
    For $n\geq 5$, the alternating group $\alt(n)$ does not have any non-canonical CL sets.
\end{thm}
\begin{proof}
    For $n\geq 6$, the set $\Phi$ for $\sym(n)$ consists of all irreducible representations that correspond to a partition in which the first part has size at most $n-4$. Thus for $n \geq 6$
    the irreducible representation $\ind^{\sym(n)}(\one_{\alt(n)}) - \one $, which corresponds to the all ones partition, is in $\Phi$. So if $L$ is a non-canonical CL set in $\alt(n)$, then by Theorem~\ref{thm:lift} this $L$ can be lifted to $\tilde{L}$, a CL set of $\sym(n)$. Since $\sym(n)$ has no non-canonical CL sets, $\tilde{L}$ must be canonical. This means that a star $S_{\alpha,\beta}$ is contained in $\tilde{L}$, but then $S_{\alpha,\beta} \cap \alt(n)$ is a star that is contained in $L$. This contradicts the fact that $L$ is non-canonical.

     The group $\alt(5)$ has no non-canonical sets, this was proven using a computer search~\cite{PalmarinThesis}.
\end{proof}

The above statement does not hold for $n=4$, since $\alt(4)$ is a Frobenius group. 

\section{CL sets in $\psl(2,q)$}\label{sec:CLinPSL}

Let $G$ be the group $\psl(2,q)$ with $q$ odd, acting on the points of a projective line. We now use the switching technique, see Theorem \ref{switch}, to construct a new non-canonical CL set $\Bar{L}$ from a known, canonical CL set $L$ in the group $G$.
We start with fixing some notation.
\begin{notation} 
\begin{enumerate}
    \item  $G_{0}:=\{g\in G  : g(0)=0 \}$,
    \item  $G_{0,\infty} :=\{g\ :\ g(0,\ \infty)=(0,\ \infty)\}$,
    \item  $S$ is the set of non-zero squares in $\bb{F}_{q}$, 
    \item  $N$ is the set of non-squares in $\bb{F}_{q}$.
    \item $L_s:=\{g \in G\ :\ g(0) \in S\}$,
    \item $L_n:=\{g \in G\ :\ g(0) \in N\}$.
\end{enumerate}
\end{notation}
We can easily verify that the sets $L_s$ and $L_n$ are canonical CL sets in $G$, as they are the union of stars. In Definition \ref{nicesets} and the subsequent paragraphs, we define the sets $A\subset L_s$ and $B\nsubseteq L_s$ that satisfy the conditions of Theorem \ref{switch}, and that give rise to a new non-canonical CL set $\Bar{L} = (L_s\setminus A)\cup B$ in $G$.


The motivation to investigate CL sets in this group $G$ came from Palmarin's thesis~\cite{PalmarinThesis} where non-canonical CL sets are found by computer search for small groups. In this thesis, a non-canonical CL set in $\psl(2,11)$ with parameter $x=5$ is found, and a computer search showed that there are no non-canonical CL sets with a smaller parameter.
For the groups $\psl(2,q)$ with $q \in \{7, 19, 23, 27\}$, the computer search timed out before any results were returned. When $q=7,11$, the computational results in~\cite{PalmarinThesis} indicate that the non-canonical CL sets of parameter $(q-1)/2$ are the union of cosets of $G_{0,\infty}$ in $G$. In this regard, for the initial canonical CL set $L_s$ we took a union of $(G_{0,\infty},G_{0})$-double cosets in $G$. Moreover,elementary calculations indicate that $L_s$ and $L_n$  are the only $(G_{0,\infty},G_{0})$-double cosets in $G$ of size $|G_{0}| (q-1)/2$.

Given $\Omega \subset (\bb{F}_{q}\cup \infty)\times (\bb{F}_{q}\cup \infty)$, we denote $A_{\Omega}:=\{g \in G\ :\ (g(0),g(\infty)) \in \Omega\}$.  Note that $A_{\Omega}$ is a union of left cosets of $H$. We now define the following sets which index certain unions of $(G_{0,\infty},G_{0,\infty})$-double cosets in $L$.  

\begin{mydef}\label{nicesets}
\begin{enumerate}
    \item $\Pi :=\{ (s,t) \ : \ s \in S,\ t\in N, \ t-s \in S\}$
    \item $\Gamma := \{(s,t) \ :\ s, t \in S, \  t-s \in S \}$,
    \item $\Lambda := \{(s,t) \ :\ s, t \in S, \  t-s \in N \}$
    \item $\Theta    := \{(s,t) \ :\ s \in S,\ t\in N, \ t-s \in N\}$, 
    \item $E_{0}:=\{(s,0) \ :\ s \in S\}$, 
    \item $E_{\infty}:=\{ (s,\infty) \:\ s \in S\}$.
\end{enumerate}
\end{mydef}     

Let $\zeta$ be an arbitrary, but fixed non-square in $\bb{F}_{q}$.
Given $\Omega \subset (\bb{F}_{q}\cup \{\infty\})\times (\bb{F}_{q}\cup \{\infty\})$, we  define two sets 
\[
\tilde{\Omega}=\{(a,b) \ : \ (b,a) \in \Omega\}, \qquad -\Omega:=\{(a,b)  \ : \ (\zeta a,\zeta b) \in \Omega\},
\]
clearly both sets are contained in $(\bb{F}_{q}\cup \{\infty\})\times (\bb{F}_{q}\cup \{\infty\})$.

We now observe that

\begin{align}
L_s &= A_{\Pi} \cup A_{\Gamma} \cup A_{\Theta} \cup A_{\Lambda} \cup A_{E_{0}} \cup A_{E_{\infty}}, \label{L}\\
L_n &= A_{-\Pi} \cup A_{-\Gamma} \cup A_{-\Theta} \cup A_{-\Lambda} \cup A_{-E_{0}} \cup A_{-E_{\infty}} \label{K}.
\end{align}

Given sets $A,B$, by $\ch_{A\shortto B}$ we denote the characteristic vector of $G_{A\shortto B}:=\{g : g(A) \subset B\}$. 

\begin{thm}\label{exoticCLset}
When $q \equiv 3 \pmod 4$, the set $\tilde{L}=(L_s \setminus A_{\Pi}) \cup A_{-\Pi}$ is a CL set; and when $q \equiv 1 \pmod 4$, the set $\tilde{L}=(L_s\setminus A_{\Lambda}) \cup A_{-\Lambda}$ is a CL set.
\end{thm}
\begin{proof}
We first prove the case of $q\equiv 1 \pmod{4}$, the case $q\equiv 3 \pmod{4}$ is very similar and we will only give an outline of it.

To show that $\tilde{L}$ is a CL set, we prove that the following equation is true:
\begin{equation}\label{eq:LC1}
    2\ch_{\tilde{L}}= \ch_{0\shortto S}+\ch_{S\shortto \infty}+\ch_{\infty \shortto N}+\ch_{\infty \shortto \infty}-\ch_{N \shortto 0} -\ch_{0 \to 0}.
\end{equation}
This shows $\ch_{\tilde{L}}$ is a linear combination of the canonical CL sets which means that $\tilde{L}$ a CL set.
The coefficient of $g \in G$ on the right-hand-side of the above equation is 
\[
c_{g}:=\delta_{S}(g(0))+\delta_{S}(g^{-1}(\infty)) + \delta_{N}(g(\infty))+\delta_{\infty}(g(\infty))-\delta_{N}(g^{-1}(0))-\delta_{0}(g(0)).
\]
We now prove that $c_{g} =2$ if $g \in \tilde{L}$ and 0 otherwise.

\paragraph{Case 1: $\{g(0),g(\infty)\} \cap \{0,\ \infty\}=\emptyset$ and $g(\infty)-g(0) \in N$.}  \hfill

In this case, there exists $s,t \in \bb{F}_{q}\setminus \{0\}$ and $\epsilon \in N$ such that $g :=\begin{pmatrix}1 & \epsilon \\ s & \epsilon t \end{pmatrix}$. We have $g(0)=s$, $g(\infty)=t$, $g^{-1}(0)=\frac{-s}{\epsilon t}$ and $g^{-1}(\infty)=-\epsilon ^{-1} \in N$.  Thus, we have  
\[
c_{g}=\delta_{S}(g(0))+\delta_{N}(g(\infty))-\delta_{N}(g^{-1}(0)).
\]

\subparagraph{Subcase 1: $s\in N$.} If $s\in N$, then $\delta_{S}(g(0))=0$. As $s,\epsilon \in N$, we have $g^{-1}(0)=\frac{-s}{\epsilon t} \in N$ if and only if $g(\infty)=t \in N$. Thus in this case, $c_{g}=0$.  
\subparagraph{Subcase 2: $s\in S$.} If $s\in S$, then $\delta_{S}(g(0))=1$. As $s \in S$ and $\epsilon \in N$, we have $g^{-1}(0)= \frac{-s}{\epsilon t} \in N$ if and only if $g(\infty)=t \in S$. Thus in this case, we have 
$c_{g}=  2 $ if $g(\infty) \in N$, and 0 otherwise.

\paragraph{Case 2: Assume $\{g(0),g(\infty)\} \cap \{0,\ \infty\}=\emptyset$ and $g(\infty)-g(0) \in S$.} \hfill

In this case, there exists $s,t \in \bb{F}_{q}\setminus \{0\}$ and $\epsilon \in S$ such that $g :=\begin{pmatrix}1 & \epsilon \\ s & \epsilon t \end{pmatrix} $.
We have $g(0)=s$, $g(\infty)=t$, $g^{-1}(0)=\frac{-s}{\epsilon t}$ and $g^{-1}(\infty)=-\epsilon ^{-1} \in S$. Thus, we have  
\[
c_{g}=\delta_{S}(g(0))+1+\delta_{N}(g(\infty))-\delta_{N}(g^{-1}(0)).
\]

\subparagraph{Subcase 1: $s\in N$.} If $s\in N$, then $\delta_{S}(g(0))=0$. As $s \in N$ and $\epsilon \in N$, we have $g^{-1}(0)=\frac{-s}{\epsilon t} \in N$ if and only if $g(\infty)=t \in S$. Thus in this case, we have 
$c_{g}= 2 $ if $g(\infty) \in N$, and 0 otherwise.

\subparagraph{Subcase 2: $s\in S$.} If $s\in S$, then $\delta_{S}(g(0))=1$. As $s,\epsilon \in S$, we have $g^{-1}(0)=\frac{-s}{\epsilon t} \in N$ if and only if $g(\infty)=t \in N$. Thus in this case, $c_{g}=2$. \\

The coefficients $c_{g}$ for other classes of $g$ can be found in a similar manner. We illustrate the values of $c_{g}$ in Table~\ref{tab:LC1}. This shows that $c_{g}=2\delta_{\tilde{L}}(g)$ and thus \eqref{eq:LC1} is true. Therefore $\tilde{L}$ is a CL set.

\begin{table}[ht]
    \centering
    \begin{tabular}{|c|c|c|c|c|}
    \hline
     $g(0)$& $g(\infty)$ & $g(\infty)-g(0)$ & $c_{g}$  \\
         \hline
          $S$  & $S$ & $S$ & 2\\
          $S$  & $S$ & $N$ & 0\\
          $S$  & $N$ & $S$ & 2\\
          $S$  & $N$ & $N$ & 2\\
          $S$  & $0$ & $\_$ & 2\\
          $S$ & $\infty$ &$\_$ & 2\\
          $N$  & $N$ & $N$ & 0\\
          $N$  & $N$ & $S$ &2 \\
          $N$  & $S$ & $N$ & 0\\
          $N$  & $S$ & $S$ & 0\\
          $N$  & $0$ & $\_$ & $0$\\
          $N$ & $\infty$ & $\_$ & $0$\\
           $0$ & $\bb{F}_{q} \setminus\{0\} \cup\{\infty\}$ & $\_$ & $0$\\
        $\infty$ & $\bb{F}_{q}$ & $\_$ & $0$\\
          \hline
    \end{tabular}
    \caption{$q\equiv 1\pmod{4}$}
    \label{tab:LC1}
\end{table}

Using essentially the same argument, \textsl{mutatis mutandis}, as in the case of $q\equiv 1 \pmod{4}$, we can prove part (i) of the theorem. In particular, when $q\equiv 3\pmod{4}$, we have 
\begin{equation}\label{eq:LC3}
    2\ch_{\tilde{L}}= \ch_{0\shortto S}+\ch_{S\shortto \infty}+\ch_{\infty \shortto S}+\ch_{\infty \shortto \infty}-\ch_{S \shortto 0} -\ch_{0 \to 0}.
\end{equation}
Let $c_{g}$ be the coefficient of $g$ in the right-hand-side. The values of $c_{g}$ in this case are illustrated in Table~\ref{tab:LC3}.

\begin{table}[ht]
    \centering
    \begin{tabular}{|c|c|c|c|c|}
    \hline
     $g(0)$& $g(\infty)$ & $g(\infty)-g(0)$ & $c_{g}$  \\
         \hline
          $S$  & $S$ & $S$ & 2\\
          $S$  & $S$ & $N$ & 2\\
          $S$  & $N$ & $S$ & 0\\
          $S$  & $N$ & $N$ & 2\\
          $S$  & $0$ & $\_$ & 2\\
          $S$ & $\infty$ &$\_$ & 2\\
          $N$  & $N$ & $N$ & 0\\
          $N$  & $N$ & $S$ & 0\\
          $N$  & $S$ & $N$ & 2\\
          $N$  & $S$ & $S$ & 0\\
          $N$  & $0$ & $\_$ & $0$\\
          $N$ & $\infty$ & $\_$ & $0$\\
          $0$ & $\bb{F}_{q}\setminus\{0\} \cup\{\infty\}$ & $\_$ & $0$\\
        $\infty$ & $\bb{F}_{q}$ & $\_$ & $0$\\
          \hline
    \end{tabular}
    \caption{$q\equiv 3\pmod{4}$}
    \label{tab:LC3}
\end{table}

\end{proof}  

\section{CL subgroups}\label{sec:CLsubgroup}

For a group $G \leq \sym(\Omega)$, a subgroup $K \leq G$ is called a \textsl{CL subgroup} if $K$ is also a $(G, \Omega)$-CL set. It is clear that the stabilizer of a point is a CL subgroup, these are considered to be \textsl{canonical CL subgroups}, as is any CL subgroup that is the union of cosets of the stabilizer of a point. 
In this section we will give examples of non-canonical CL subgroups. But first we will show that it is not difficult to determine if a subgroup is a CL subgroup.

\begin{lemma}\label{lem:CLgroups}
Let $G$ be a group and $H \leq G$.
A subgroup $K \leq G$ is a CL subgroup under the action of $G$ on $G/H$ if and only if every irreducible representation that is a constituent of $\ind^G(\one_K)$ is also a constituent of $\ind^G(\one_H)$. 
\end{lemma}
\begin{proof}
Since $K$ is a subgroup, by Corollary~\ref{cor:CLsubform}, $K$ is a CL set if and only if
$\phi(K) = \sum_{x \in K}\phi(x) = 0$ for all irreducible representations $\phi$ not in $\ind^G(\one_H)$. 
Since
\[
\sum_{x \in K}\phi(x) = |K| \langle \one_K, \res_K(\phi) \rangle_K =  |K| \langle \ind^G(\one_K) , \phi \rangle_G,
\]
$\phi(K) = 0$ if and only $\phi$ is not a constituent of $\ind^G(\one_K)$.
Thus $K$ is a CL set if and only if every irreducible representation in $\ind^G(\one_K)$ is also in $\ind^G(\one_H)$.
\end{proof}

There is a simple situation where there are clearly canonical CL subgroups. Assume $H \leq K \leq G$, and consider the action of $G$ on $G/H$. In this case, every irreducible representation in $\ind^G(\one_K)$ is also in $\ind^G(\one_H)$, so $K$ is a CL subgroup, but $K$ is the union of cosets of $H$, so $K$ is a canonical CL subgroup. This will only occur exactly when $H$ is not a maximum subgroup.

If the action of $G$ is 2-transitive, then $\ind^G(\one_{G_\omega})$ is the sum of only two irreducible representations, so any non-canonical CL subgroup $H$ must have
 $\ind^G(\one_{H}) = \ind^G(\one_{G_\omega})$, and by~\cite[Section 7]{MSi2019} must have the same inner distribution as $G_\omega$, this means that $H$ is an intersecting set. 
 
 \begin{example}
 Consider $\pgl(n,q)$, with $n>2$ and $q$ a prime power, and its action on the projective points. This action is 2-transitive, but does not have the strict EKR property as the stabilizer of a point and the stabilizer of a hyperplane are both maximum intersecting subgroups~\cite{Spiga2019}. In this case, the stabilizer of a {hyper}plane is a CL subgroup, since it has the same size of the stabilizer of a point, it is not a canonical CL subgroup. Conversely, the stabilizer of a point is a non-canonical CL set of $\pgl(n,q)$ with its action on the cosets of the stabilizer of a {hyper}plane.
 \end{example}

The symmetric group, with its natural action, does not have any non-canonical CL sets, but, as we see with our next example, it can have non-canonical CL sets under a different action.

\begin{example}\label{ex:symmetricGroup}
Consider the group $\sym(n)$ with its transitive action on $\binom{[n]}{t}$, the $t$-subsets of $[n]$. The stabilizer of this action is $H = \sym(t) \times \sym(n-t)$, so this is the action of $G$ on $G/H$.
The subgroup $K = \sym(n-1)$ also acts on $\binom{[n]}{t}$. The irreducible representation of $\sym(n)$ can be represented by integer partitions of $n$, and it is well-known that 
\begin{enumerate}
\item $\ind^{\sym(n)}(\one_{ \sym(t) \times \sym(n-t) }) = [n] + [n-1,1] + \dots + [n-t,t]$ and
\item $\ind^{\sym(n)}(\one_{ \sym(n-1) }) = [n] + [n-1,1]$
\end{enumerate}
(see, for example,~\cite[Section 12.5]{GMbook} or \cite{MR1093239}).
By Lemma~\ref{lem:CLgroups}, this shows that $\sym(n-1)$ is a CL subgroup under this action. In fact any Young subgroup $\sym(s) \times \sym(n-s)$ with $s \leq t$ is a CL subgroup. 

With this action, a minimal canonical intersecting set is a coset of $\sym(t) \times \sym(n-t)$, since in every coset there is an element that moves $n$ (provided that $t>1$), there are no cosets that are contained in $\sym(n-1)$---this implies that $\sym(n-1)$ is a non-canonical CL subgroup.
\end{example}

This previous result can be extended to any $t$-transitive group.

\begin{lemma}\label{lem:stabilizersubgroups}
Let $G$ be a $t$-transitive a group acting on a set $\Omega$ with $t\geq 2$.
Let $G_\omega$ be the stabilizer of an element $\omega \in \Omega$.
Then $G_\omega$ is a CL set of $G$ with its action on $\binom{\Omega}{t}$.
\end{lemma}
\begin{proof}

Since $G$ is $t$-transitive, with $t\geq 2$, the group is 2-transitive, so $\ind^G(\one_{G_\omega})$ consists of exactly two irreducible representations $\one$ and $\phi$. By Lemma~\ref{lem:CLgroups}, it suffices to prove that $\phi$ is a constituent of $\ind^G(\one_{G_T})$, where $G_T$ is the set-wise stabilizer of a $t$-set. 

This can be seen with Frobenius reciprocity,
\[
\langle \ind^G(\one_{G_T}),  \ \ind^G(\one_{G_\omega}) \rangle_G 
= \langle \res_{G_\omega} ( \ind^G(\one_{G_T}) ), \ \one \rangle_{G_\omega} 
\geq 2.
\]
The last equation follows since $G_\omega$ has at least 2 orbits on the $t$-sets.
\end{proof}

\begin{example}\label{ex:Alt5}
Consider the action of $\alt(5)$ on the pairs from $\binom{5}{2}$. This is a transitive action with degree 10. The stabilizer of a point is the group $(\sym(2) \times \sym(3)) \cap \alt(5)$. This is a group that does not have the EKR property since the group $\alt(4)$ is also an intersecting set with size twice that of the stabilizer of a point. From the previous example, $\alt(4)$ is a CL subgroup, but there is another CL subgroup, the group $D_5$. 

The irreducible representations of $\alt(5)$ includes $\phi$ and $\psi$, which are representations with degree $4$ and $5$ respectively. The representation induced by $D_5$ is $\one + \psi$, and the representation induced by $\alt(4)$ is $\one + \phi$. Since the representation induced by  $(\sym(2) \times \sym(3)) \cap \alt(5)$ is $ \one + \phi + \psi$, the representations also show that $D_5$ and $\alt(4)$ are both CL subgroups.

Similarly, the subgroups $ (C_3 \times C_3) \rtimes C_4$ in $\alt(6)$, and $\psl(3,2)$ in $\alt(7)$ are also CL subgroups.
\end{example}

\begin{example}\label{ex:sublines}
Let $q$ be a power of an odd prime.
 Consider the group $G:=\psl(2,q^{2})$. It is well-known (see~\cite{king2005subgroup}) that $G$ has two non-conjugate copies of $\pgl(2,q)$ as maximal subgroups. Let $H\leq G$ be such that $H \cong \pgl(2,q)$. Consider the transitive action of $G$ on $G/H$. Let $K\leq G$ be such that $K \cong \bb{F}_{q^2} \rtimes S$, where $S$ is the group of non-zero squares in $\bb{F}_{q^2}$. It is well-known that the action of $G$ on $G/K$ is $2$-transitive. Let $\psi \in Irr(G)$ be such that $1+\psi= \ind^G(\one_{K})$. Now $\psi$ is the unique character of degree $q^2$ of $\psl(2,q^2)$. It is not too difficult to see from the character table that $\psi(H)=|H|$. Therefore, by Lemma~\ref{lem:CLgroups}, $K$ is a CL subgroup with respect to the action of $G$ on $G/H$. As $|K|/|H|=\frac{q}{2}$ is not an integer, this is an CL set with fractional parameter and thus not a canonical one.
 
\end{example}

We now consider some permutation actions of $\pgl(2,2^k)$ which are not $2$-transitive. Through these, we will see that CL sets may not behave as well when we move away from $2$-transitive actions.

\subsection{Some non-canonical CL sets in ${\pgl(2,2^k)}$.}\label{sec:examplePGL}

In the next two examples $G := \pgl(2,q)$ where $q = 2^k$ with various different actions. As $q$ is even, $\psl(2,q) \cong \pgl(2,q)$. The subgroup structure of this group and its character table are well-known, see \cite{P-S} for details, we record it in Table~\ref{ch1}.
The subgroup lattice of $G$ is well-known and can be found in many classical group theory books such as~\cite{huppert2013endliche}. We refer the reader to~\cite{king2005subgroup} for a modern exposition. 

\begin{table}
\caption{Character table of $\pgl(2,q)$, for $q$ even.}
\centering
\begin{tabular}{c |c| c c c c }
\hline
& class type & $1$ & $u$ & $d_{x}$ & $v_{r}$\\
& Number & $1$ & $1$ & $\frac{q}{2}-1$ & $\frac{q}{2}$\\
& class size & $1$ & $q^{2}-1$ & $q(q+1)$ & $q(q-1)$\\
\hline
Character type & number &  &  &  & \\
\hline
$\one$ & $1$ & $1$ & $1$ & $1$ & $1$\\
$\psi_{1}$ & $1$ & $q$ & $0$ & $1$ & $-1$\\
$\eta_{\beta}$ & $\frac{q}{2}$ & $q-1$ & $-1$ & $0$ & $-\beta(r\bb{F}_{q}^{*})-\beta(r^{-1}\bb{F}_{q}^{*})$\\
$\nu_{\gamma}$ & $\frac{q}{2}-1$ & $q+1$ & $1$ & $(\gamma(x)+\gamma(x^{-1}))$ & $0$ \\ \hline
\end{tabular}
\label{ch1}
\end{table}

\begin{example}\label{fractional}
Let $q=2^{k}$ and consider the action $G$ on $\binom{\PG(1,q)}{2}$ (the unordered pairs of points). 
This action is equivalent to the action of $G$ on $G/H_{q-1}$ where $H_{q-1}$ is the stabilizer of the set $\{[1, 0], \ [0, 1]\}$ ($H$ is isomorphic to a dihedral group of order $2(q-1)$). The decomposition of $\ind^G(\one_{H_{q-1}})$ is 
\[
\ind^G(\one_{H_{q-1}}) = \one + \psi_1 + \sum_{\gamma} \nu_\gamma. 
\]
By Lemma~\ref{lem:stabilizersubgroups} and using the character table of $G$, we see that $K=G_{[1, 0]}$, a point stabilizer, has $\ind^G(\one_K) = \one + \psi_1$. So, by Lemma~\ref{lem:CLgroups}, $K$ is also CL subgroup under the action on pairs of points.  Since both $H_{q-1}$ and $K$ are maximal subgroups, $K$ is not a disjoint union of cosets of conjugates of $H_{q-1}$.  This is an example of the type of CL subgroup in Lemma~\ref{lem:stabilizersubgroups}.

Again by Lemma~\ref{lem:stabilizersubgroups} and the character table of $G$, we can see that $F\cong \bb{F}_{q}$ (the Sylow-2-subgroup of $\pgl(2,q)$) has
\[
\ind^G(\one_F) = \one + \psi_1 + 2 \sum_{\gamma} \nu_\gamma,
\]
so it is also a CL subgroup under this action. The size of $F$ is $q$, so this is an example of a fractional CL set, as $F \leq K$ it also shows that $K$ is not a minimal CL set. 
It is straight-forward to see that $F$ is a minimal CL set, just consider any proper subset $A$ of $F$. Any of the characters $\eta_{\beta}$, has $\sum \eta_{\beta}(a)\neq 0$, so $A$ is not a CL set. We note that $F$ is a CL set of parameter less than $1$. 
\end{example}

\begin{example}
In this example, consider the action of $G=\pgl(2,q)$ on $G/F$ where $F\cong \bb{F}_{q}$.
Again, take $H_{q-1}$ to be the subgroup isomorphic to a dihedral group of order $2(q-1)$.
The decompositions given in the previous example and Lemma~\ref{lem:stabilizersubgroups} imply that $H_{q-1}$ is a CL subgroup under this action. Note that $H_{q-1}$ is not an independent set in the derangement graph corresponding to this action.
\end{example}

\begin{example}
 Let $d$ be a proper divisor of $q+1$ and let $H_{d}$ be a subgroup isomorphic to the dihedral group of size $2d$. Consider the action of $G=\pgl(2,q)$ on $G/H_d$. Using the character table of $G$ and Lemma~\ref{lem:stabilizersubgroups}, we can show that any Sylow-2-subgroup of $G$ is a CL subgroup under this action. Such a CL subgroup is non-canonical as $q$ is not divisible by $2d$. 
\end{example}

\begin{remark}
As the subgroup lattice of $\pgl(2,2^k)$ is known, it would not be too difficult to find all the CL subgroups corresponding to all permutation actions of $\pgl(2,2^k)$.
\end{remark}

\section{CL sets from the union of non-disjoint canonical CL sets}\label{sec:CLfromnondisj}

It is clear that a union of disjoint canonical CL sets forms a CL set, but in this section we will see that it is also possible to form CL sets from canonical CL sets that have non-empty intersection.

\begin{lemma}\label{lem:CLunions}
Let $G$ be a group acting on a set $\Omega$. Assume there exists a set of pairs 
\[
J = \{ (a_1,b_1),(a_2,b_2), \dots, (a_\ell, b_\ell) \},
\] where $a_i, b_i \in \Omega$ for $i \in 1,\dots, \ell$,
such that for every element $g \in G$, there are either exactly $c$ pairs in $J$ with $g(a_i) = b_i$ or there are no such pairs. Let $S \subseteq G$ be the set of all elements for which there are exactly $c$ pairs in $J$ with $g(a_i) = b_i$. Then $S$ is a CL set. 
\end{lemma}
\begin{proof}
The characteristic vector of $S$ is 
\[
\ch_{S} =  \frac{1}{c}  \sum_{i=1}^\ell \ch_{a_i, b_i}
\]
so $S$ is a CL set.
\end{proof}

The CL sets in the above lemma are unions of stars, but these stars need not be disjoint. The CL set is canonical if and only the the stars are disjoint. Many such examples can be found from imprimitive groups. If $G$ is an imprimitive group acting on $\Omega$ with blocks $B_1, B_2, \dots, B_\ell$, then the set of all elements in $G$ that map every element in $B_i$ to $B_j$ is the union of all the stars $G_{a \to b}$ with $a \in B_i$ and $b \in B_j$. 
In this case, using the notation in Lemma~\ref{lem:CLunions}, $J=\{ (b_i, b_j) \ : \ b_i \in B_i, b_j \in B_j \}$, and for each $g \in G$ there are either $|B_i|$ pairs in $J$ with $g(b_i) = b_j$, or none. But these examples are canonical CL sets, since for a fixed element $b \in B_i$, this set is the union of all the disjoint stars $\cup_{b_j \in B_j }G_{b \to b_j}$.

In this section we will consider examples of CL sets of the type in Lemma~\ref{lem:CLunions}, that are not unions of disjoint stars.

\begin{example}
The CL subgroup in Example~\ref{ex:symmetricGroup} is also an example of a CL set of the type described in Lemma~\ref{lem:CLunions}. The action of $\sym(n-1)$ on the cosets $\sym(n) / (\sym(t) \times \sym(n-t) )$ is not transitive, it has two orbits, the $t$-sets with $n$, call these $\mathcal{A}$, and the sets without $n$, call these $\mathcal{B}$. No element of $\sym(n-1)$ maps an element in $\mathcal{A}$ to an element in $\mathcal{B}$. For every element $\sigma$ that is in $\sym(n)$, but not in $\sym(n-1)$, if $\sigma^{-1}(n) \not \in A \in \mathcal{A}$, then $n\not \in \sigma(A)$, so $\sigma$ maps $A$ to a set in $\mathcal{B}$. Since there are $\binom{n-2}{t-1}$ sets in $\mathcal{A}$ that do not contain $\sigma^{-1}(n)$, each permutation in $\sym(n)$, that is not in $ \sym(n-1)$, maps $\binom{n-2}{t-1}$ elements from $\mathcal{A}$ to $\mathcal{B}$.
\end{example}

\begin{example}
In Example~\ref{ex:Alt5}, under the action of $\alt(5)$ on $\binom{[5]}{2}$, $D_{10}$ is a CL subgroup. This is also an example of a CL sugroup of the type in Lemma~\ref{lem:CLunions}. To see this, let $\mathcal{A}$ be the set of edges in a 5-cycle, and $\mathcal{B}$ the set of edges in its complement. The sets $\mathcal{A}$ and $\mathcal{B}$ are the orbits of $D_5$ on $\binom{[5]}{2}$.
Then the set $D_5$ is a CL set, since any permutation from $D_5$ maps no set in $\mathcal{A}$ to a set in $\mathcal{B}$, while every other permutation in $\alt(5)$, maps exactly $3$ of the sets in $\mathcal{A}$ to a set in $\mathcal{B}$. The group $D_5$ does not contain a star under this action.
\end{example}

\begin{example}
 Let $G$, $H$, $K$ be as in Example~\ref{ex:sublines}. The number of orbits for the action of $K$ on $G/H$ is the same as the number of $(K,H)$-double cosets in $G$, which in turn is the same as the number of orbits for the action of $H$ on  $K \backslash G$ (space of right cosets of $K$). Using $\ind^G(\one_{K})=1+ \psi$ and the Orbit-Counting formula, we see that the action of $H$ on $K\backslash G$ has exactly 2 orbits, and thus the same is true for the action of $K$ on $G/H$. Let $A$ and $B$ be the $K$-orbits on $G/H$. As $K$ is a maximal subgroup, we have $K=\{g\ :\ g(A)=A\}=\{g\ :\ g(B)=B\}$. Therefore, given $g \in K^c$, there exists $(a,b)\in A\times B$ such that $g(a)=b$. Therefore, we have $K^c= \bigcup\limits_{(a,b)\in A\times B}G_{a \shortto b}$. As $K$ is a CL subgroup, $K^c$ is a CL set which is a union of non-disjoint stars. As $\frac{|K^c|}{|H|} =\frac{|G|}{|H|}-\frac{q}{2}$ is not an integer, $K^c$ is not a canonical CL set.
\end{example}

\subsection{CL sets in the 2-transitive action $PSL(2,q^2)$}\label{subsec:psl(2,q^2)}

In this subsection we construct a non-canonical CL set for $G=\psl(2,q^2)$ where $q$ is a power of an odd prime. It is well-known that $G$ acts 2-transitively on the set of the $q^2+1$ projective points, which we denote by $\PG(1,q^2)$. Since $q$ is odd, $q^2 \equiv 1 \pmod{4}$ so it has 2 orbits on the 3-sets of distinct points. Further, $\PG(1,q)$ is a subline  of $\PG(1,q^2)$ and the set-wise stabilizer of $\PG(1,q)$ is isomorphic to $\pgl(2,q)$. Any set of three points from $\PG(1,q)$ are contained in a single orbit under the action of $G$, since this action has two orbits, there is a second subline corresponding to 3-sets in the other orbit and no element of $G$ maps the first subline to the second. We call these two sublines $\Omega_1$ and $\Omega_2$, both are $(q+1)$-sets of points from $\PG(1,q^2)$. The set-stabilizers of $\Omega_1$ and $\Omega_2$ are both isomorphic to $\pgl(2,q)$, but these two subgroups are not conjugate in $\psl(2,q^2)$. 
The next result shows that for any element from $g \in G$, the size
$| g (\Omega_1) \cap \Omega_2 |$ can only take two different values.

\begin{lemma}\label{int}
Let $G = PSL(2,q^2)$ where $q$ is an odd prime power. 
Let $\Omega_1$ and $\Omega_2$ be two sublines, $\PG(1,q)$, from different orbits (as above).
For any $g \in G$ 
\[
| g ( \Omega_1)  \cap \Omega_2 | \in \{0,2\}.
\]
\end{lemma}    
\begin{proof}
Since three distinct points determine a line, and the triples from $\Omega_1$ and $\Omega_2$ are in distinct orbits, it must be that $| g(\Omega_1)  \cap \Omega_2 | \in\{0,1,2\}$ for every $g \in G$. So we only need to show that the intersection cannot be equal to one. Since $\Omega_1 = \PG(1,q)$, we may assume that it contains both $[1,0]$ and $[0,1]$. Further, $\Omega_2$ can also be assumed to contain $[1,0]$ and $[0,1]$.

Assume that $| g (\Omega_1)  \cap \Omega_2 |  \geq 1$, we will show that this intersection actually has two elements. Without loss of generality we can assume the point $[1,0]$ is in this intersection. So there is an $\omega \in \Omega_1$ with $g(\omega) = [1,0]$. Further, there is an $h \in \pgl(2,q)$ with $h(\omega) = [1,0]$. So $gh$ maps $[0,1]$ to $[0,1]$, and, if it is the case that $| gh (\Omega_1)  \cap \Omega_2 |  = 2$, then $| g (\Omega_1)  \cap \Omega_2 |  = 2$, since $h$ fixes $\Omega_1$. Thus we may assume that 
$g \in G_{0}=\begin{pmatrix}
1 & u \\
0 & s
\end{pmatrix}$ for some $u \in \bb{F}_{q^2}$ and some $s$ a square element in  $\mathbb{F}_{q^2}\setminus \{0\}$. 
If $u \in \bb{F}_{q}$, then $[1, -u^{-1}] \in \Omega_1$ and $g([1, -u^{-1}])= [0,1]$ and therefore $[0,1] \in g(\Omega_1)  \cap \Omega_2$. 

Alternately, consider the case where $u \in \bb{F}_{q^2} \setminus \bb{F}_{q}$, in this case $u$ can be expressed as the sum of an element in $\bb{F}_{q}$ and an element from $\bb{F}_{q^2}  \setminus \bb{F}_{q}$. For any $r \in \bb{F}_{q^2}  \setminus \bb{F}_{q}$, and any $b \in \bb{F}_{q} \setminus \{0\}$, the element $r^{-1}b$ is not in $\bb{F}_{q}$. In particular, there is an $r \in \bb{F}_{q^2}  \setminus \bb{F}_{q}$ and a $b \in \bb{F}_{q} \setminus \{0\}$ so that $[1, rb^{-1}] \in \Omega_2$. Further, there exists $a \in \bb{F}_{q}$ such that $u=a+sr^{-1}b$. 

If $a=0$, then
\[
  g([0,1]) =  [u, s] = [1, u^{-1}s ] = [1, rs^{-1}b^{-1} s] = [1, rb^{-1}],
 \]
 so $[0,1]$ is the second element in  $g (\Omega_1) \cap \Omega_2$.
 
 If $a\neq 0$, then 
\[
    g([1,-a^{-1}]) = [1-a^{-1}u, -sa^{-1}] = [ a -u, -s] = [-bsr^{-1}, -s] =  [br^{-1}, 1] = [1, rb^{-1}],
\]
and $[1, rb^{-1}]$ is the second element in  $g(\Omega_1)  \cap \Omega_2$.
\end{proof}

This result, with Lemma~\ref{lem:CLunions}, shows that there are CL sets in $\psl(2, q^2)$.

\begin{thm}\label{sublineCLset}
Let $G = \psl(2,q^2)$ where $q$ is a prime power.
Let $A$ and $B$ be any two sublines that lie in different $G$ orbits. Then 
$L :=\{ g \in G \ : \  | g(A) \cap B |=2\}$ is a CL set for $G$ with its action on projective points.
\end{thm}
\begin{proof}
Define $J=\{ (x,y) \in \Omega_1 \times \Omega_2 \}$, then by Lemmas~\ref{lem:CLunions} and~\ref{int} the set 
\[
L :=\{ g \in G \ : \  | g( \Omega_1 ) \cap \Omega_2  |=2\}
\]
is a CL set for $G$.

For any two sublines $A$ and $B$ from different orbits, there are permutations $h_1$ and $h_2$ in $G$ such that $h_1 (\Omega_1)=A$ and $h_2(\Omega_2) =B$ ($h_1$ is from one copy of $\pgl(2,q)$, and $h_2$ is an element from another, non-conjugate, copy of $\pgl(2,q)$). Thus
\[
|g(A) \cap B| = |  h_1g(\Omega_1)  \cap  h_2(\Omega_2)   | =  |  h_1gh_2^{-1} (\Omega_1)  \cap  \Omega_2 |.
\]
By Lemma~\ref{int} this intersection is either 0 or 2 for every $g$, and the result follows from Lemma~\ref{lem:CLunions}. 
\end{proof}

\begin{lemma}\label{size} 
The set $L$ from Theorem~\ref{sublineCLset} is a non-canonical CL set with $ | L |= \dfrac{q^2 (q+1)^2(q^2-1)}{4} $.
\end{lemma}
\begin{proof}
It suffices to prove this for $\Omega_1$ and $\Omega_2$, rather than for all $A$ and $B$.
First set $\alpha=\{g \in G \ : \ g ( \Omega_1 ) \cap \Omega_2   =\{ [1,0] ,\ [0,1] \} \}$. 
By the $2$-transitivity of $\psl(2,q^2)$, and the fact that the number of 2-subsets of $\Omega_2$ is $\binom{q+1}{2}$, we have $|L|= \binom{q+1}{2} \times \alpha$.

To find $\alpha$, note that there are $(q+1)q$ possible pairs of points $k_0, k_{\infty}\in \Omega_1$ that can be mapped to $\{0, \infty\}$. Then an element $g\in G$ with $g(k_0) = 0$ and $g(k_\infty) = \infty$ has the form $\begin{pmatrix}
    1 & -k_\infty^{-1}\\ \ell & -\ell k_0^{-1}
\end{pmatrix}$. Since there are $\frac{q^2-1}{2}$ possibilities for the element $\ell \in \mathbb{F}_{q^2}\setminus \{0\}$, we have that $\alpha = \binom{q+1}{2}q(q+1)$ and
\[
|L| = \binom{q+1}{2} \times \alpha   = \frac{q^2(q+1)^2(q^2-1)}{4}.
\]

Finally, we need to show that $L$ is not a canonical CL set. If it were canonical, then $L$ must be a disjoint union of $\frac{(q+1)^2}{2}$ sets of the form $G_{x\shortto y}$. 
By $2$-transitivity of $G$, we have $G_{x \shortto y}$ and $G_{u \shortto w}$ are disjoint if and only if either $x=u$ or $y=w$. 
Also, if $x\neq u$ and $y \neq w$, then $|G_{x\shortto y} \cap G_{u\shortto w}|= \frac{q^2-1}{2}$. 

Since each element in $L$ maps exactly two elements in $\Omega_1$ to $\Omega_2$ we can easily write its characteristic vector as a linear combination of stars: $\ch_{L}=\frac{1}{2} \sum\limits_{(x,y)\in \Omega_1 \times \Omega_2 } \ch_{G_{x\shortto y}}$.
Further, 
\[
\ch_{G_{u \to w}}^{T} \ch_{L} =
\begin{cases}
    |G_{u \to w}|=\frac{q^{2}(q^2-1)}{2}, & (v,w) \in \Omega_1 \times \Omega_2; \\
    0 & \textrm{otherwise.}
\end{cases}
\]
Thus $L$ contains at most $q+1$ disjoint stars. 

Since the set $\bigcup\limits_{k \in\Omega_2 }G_{0\shortto k}$ is contained in $L$, we can consider $\tilde{L}=L \setminus \bigcup\limits_{y \in\Omega_2 }G_{0\shortto k}$. For any $(x,y) \in \Omega_1 \times \Omega_2 $, with $x\neq 0$, we compute that 
\begin{align*}
|G_{x,y} \cap \tilde{L}|&=\frac{1}{2}|G_{x,y}|+\frac{1}{2}\sum\limits_{\stackrel{(u,w) \in \Omega_1 \times \Omega_2}{u\neq x, w\neq y}}|G_{u,w} \cap G_{x,y}|- \sum\limits_{z \in \Omega_2\setminus \{y\}}|G_{x,y}\cap G_{0,z}| \\
&= \frac{q^2(q^2-1)}{4} + \frac{1}{2} \frac{q^2(q^2-1)}{2} - \frac{q(q^2-1)}{2} \\
&= (q^2-q) \frac{q^2-1}{2} \\
&< \frac{q^2(q^2-1)}{2}.
\end{align*}

So no other stars can be entirely contained in $\tilde{L}$. Hence, $\Tilde{L}$ is a proper non-canonical CL set.
\end{proof}


\section{Future Research}
We end with some directions for further work. In this article, we constructed non-canonical CL sets in some $2$-transitive groups. A natural question would be to characterize all CL sets in specific classes of $2$-transitive groups. We recall from Corollary~\ref{cortussen01} that a CL set of $2$-transitive group of parameter $1$ is a maximum intersecting set. Classification of maximum intersecting sets has been completed for the $2$-transitive actions of $\sym(n)$, $\alt(n)$, $\pgl(n,q)$, and $\psl(2,q)$ (see \cite{CK2003, LPSX2018, MS2011, Spiga2019}). As a starting point, we propose to investigate the classification of CL sets in these groups. We saw that there are no non-canonical CL sets in $\sym(n)$ and $\alt(n)$ (see Theorem~\ref{sym} and Theorem~\ref{thm:NoAlternating}). In Section \ref{sec:CLinPSL}, we constructed a family of non-canonical CL set in $\psl(2,q)$, and another family in $\pgl(2,q^2)$. Computations in \cite{PalmarinThesis} indicate that for small vales of $q$, these are the only non-canonical sets in $\psl(2,q)$. Our constructions of non-canonical CL sets in $\psl(2,q)$ do not ``lift'' to CL sets in $\pgl(2,q)$. Computations for small values of $q$ did not yield any non-canonical CL sets in $\pgl(2,q)$ which leads us to the following question.

\begin{question}
Classify CL sets in $\psl(2,q)$ and $\pgl(2,q)$. In particular, are there non-trivial CL sets in $\pgl(2,q)$?
\end{question}

As $\sym(n)$, $\alt(n)$, and $\psl(2,q)$ satisfy the strict-EKR property (see \cite{CK2003, ku2007intersecting, LPSX2018}), by Corollary~\ref{cortussen01}, any parameter one CL set in these groups must be canonical. However only $\psl(2,q)$ has non-canonical CL sets. It is natural to ask the following question.
\begin{question}
Under what conditions does a $2$-transitive group posses a proper non-canonical CL set of parameter greater than $1$?
\end{question}

From Lemma~\ref{n-clique}, we know that if the derangement graph of a $2$-transitive group has a large clique, the parameter of a CL set in such a group must be an integer. All the examples we found of CL sets in $2$-transitive group have integer parameters. The numerical condition given in Corollary~\ref{cortussen01} shows that the parameter of a CL set of a $2$-transitive group is at least $1$, but it does not rule out the existence of fractional parameter CL sets.
\begin{question}
Can $2$-transitive group posses a CL set of fractional parameter?
\end{question}

In the case of CL sets in $\mrm{PG}(3,q)$, there are many results---see, for example,~\cite{GMModularequality, Metsch2014}---that rule out the existence of CL sets of certain parameters. It would be interesting to find such numerical restraints on parameters of non-canonical CL sets in $2$-transitive groups.

\begin{question}
Find numerical conditions on the parameter of non-canonical CL sets of $2$-transitive groups.
\end{question}

\section*{Acknowledgements}
Jozefien D’haeseleer is supported by the Research Foundation Flanders (FWO) through the grant 1218522N.
Karen Meagher's research is supported by NSERC Discovery Research Grant No.: RGPIN-2018-03952.
Venkata Raghu Tej Pantangi is supported by the PIMS postdoctoral fellowship.

\bibliographystyle{abbrv}
\bibliography{ref.bib}

\end{document}